\pdfoutput=1
\documentclass[11pt]{article}

\usepackage[a4paper,margin=25mm]{geometry}
\usepackage[T1]{fontenc}
\usepackage{lmodern}
\usepackage{amsmath,amssymb,amsthm,mathtools}
\usepackage{booktabs,array,longtable,pdflscape,capt-of}
\usepackage{enumitem}
\usepackage{microtype}
\usepackage{xcolor}
\usepackage[hidelinks]{hyperref}

\definecolor{proofblue}{RGB}{26,55,92}
\hypersetup{
  pdftitle={The Primitive Generalized Fermat Equation x\textasciicircum3+y\textasciicircum5=z\textasciicircum7: A computer-assisted proof},
  pdfsubject={A computer-assisted proof},
  pdfauthor={Peter Chocian},
  pdfkeywords={generalized Fermat equation, rational points, descent, Jacobians, computer-assisted proof},
  colorlinks=true,
  linkcolor=proofblue,
  citecolor=proofblue,
  urlcolor=proofblue
}

\newcommand{\Q}{\mathbf Q}
\newcommand{\Z}{\mathbf Z}
\newcommand{\F}{\mathbf F}
\newcommand{\Pone}{\mathbf P^1}
\newcommand{\OO}{\mathcal O}
\newcommand{\m}{\mathfrak m}

\newcommand{\PP}{\mathfrak P}
\newcommand{\Sha}{\mathop{\mathrm{Sha}}}
\newcommand{\rank}{\operatorname{rank}}
\newcommand{\ord}{\operatorname{ord}}
\newcommand{\Gal}{\operatorname{Gal}}
\newcommand{\disc}{\operatorname{disc}}

\newtheorem{theorem}{Theorem}[section]
\newtheorem{proposition}[theorem]{Proposition}
\newtheorem{lemma}[theorem]{Lemma}

\theoremstyle{definition}

\theoremstyle{remark}

\title{\textbf{The Primitive Generalized Fermat Equation}\\[1mm]
       $x^3+y^5=z^7$\\[3mm]
       \large A computer-assisted proof}
\author{Peter Chocian\\[1mm]\small Independent researcher}
\date{22 September 2026}

\begin{document}
\maketitle

\begin{abstract}
We prove that the generalized Fermat equation
\[
                         x^3+y^5=z^7
\]
has no solution in nonzero coprime integers, and we make explicit how the
proof extends the prior work cited below.  Dahmen--Siksek had
established the signed local descent and eliminated the three cases in which
the associated degree-seven algebra is reducible.  Putz had then proved that
the remaining irreducible case can involve only seven septic fields: six
pure fields and one exceptional field.  The missing step was to exclude
those seven fields.

For the six pure fields we construct an explicit Fano resolvent, descend to
a smooth plane quartic of genus three over $\Q(\sqrt{-7})$, and prove that
its rational-parameter locus consists only of five points above the branch
values.  For the
exceptional field we combine the modularity and conductor results of
Pacetti--Villagra Torcomian with level lowering over $\Q(\sqrt5)$ and a
finite Hecke comparison at $29$.  We also reconstruct the three reducible-sector arguments of Dahmen--Siksek as independently replayable calculations, including
their database-free identification of the quadratic field, and add an intrinsic
formal-group treatment at the ramified prime.  The paper labels every major
step as a literature input, reconstructed input, or new argument.  All
project-specific finite calculations are supplied as exact certificates,
programs, inputs, and authenticated logs.
\end{abstract}

\begin{quote}\small
\textit{2020 Mathematics Subject Classification.}
Primary 11D41; Secondary 11G30, 11Y50, 14G05.\par
\textit{Keywords.} Generalized Fermat equations; rational points;
explicit descent; Hilbert modular forms; computer-assisted proofs.
\end{quote}

\tableofcontents

\section{Introduction: the frontier and the completion}

\subsection{The equation and the theorem}

For pairwise coprime nonzero integers, the equation
\begin{equation}\label{eq:original}
                         x^3+y^5=z^7
\end{equation}
is the generalized Fermat equation of signature $(3,5,7)$.  We shall use
the variable order
\begin{equation}\label{eq:DS-order}
                    (X,Y,Z)=(y,x,z),\qquad X^5+Y^3=Z^7.
\end{equation}
All signs are allowed.  Since a prime dividing two of $X,Y,Z$ also divides
the third, primitivity is equivalent here to pairwise coprimality.  Our main
result is the following.

\begin{theorem}\label{thm:main}
There are no nonzero primitive integers $x,y,z$ satisfying
\[
                             x^3+y^5=z^7.
\]
\end{theorem}

Solutions in which one of $x,y,z$ equals $\pm1$ are not excluded in
advance.  They are nonzero, so the parameter $\eta$ of \eqref{eq:phi}
below satisfies $\eta\notin\{0,1,\infty\}$, and they are subject to
exactly the same descent and local analysis as every other solution;
nothing in the proof assumes $|x|,|y|,|z|>1$.  Up to sign, such a
solution would consist of two consecutive integers each of which is a
perfect power with exponent in $\{3,5,7\}$, and Catalan's conjecture,
proved by Mih\u{a}ilescu \cite{Mihailescu}, excludes this independently;
the proof below does not use that theorem.

The purpose of this introduction is not only to state the theorem, but to
say precisely what was known before the present work and what had still to
be done.  This distinction matters in a long computer-assisted proof: it
makes clear which general theorems are imported, which earlier computations
are reconstructed for confidence, and which arguments close the remaining
gap.

\subsection{Prior work}

The mathematical starting point was already far advanced.  Dahmen and
Siksek proved the following strong partial result
\cite[Theorem~1]{DahmenSiksek}: every nontrivial primitive solution of
$X^5+Y^3=Z^7$ satisfies one of the alternatives
\begin{equation}\label{eq:DS-theorem-one}
 2\cdot3\cdot5\mid Z\ \text{ and }\ 7\nmid XY,
 \qquad\text{or}\qquad
 2\nmid Z, 3\nmid XYZ, 5\nmid YZ, 7\mid Y.
\end{equation}
In particular, a primitive solution for which $3\mid XY$, $5\mid Y$, or
$7\mid X$ is trivial.  Their descent introduces the degree-seven algebra
used below, proves that only four factor-degree patterns can occur, and
eliminates the three reducible patterns
\[
                         [1,6],\qquad[2,5],\qquad[3,4].
\]
The only surviving pattern is therefore the irreducible one, $[7]$.

Putz then made this residual problem finite.  His exhaustive Hunter
enumeration, with the prescribed local algebra data, proves that an
irreducible descent algebra must be one of exactly seven fields
\cite[Theorem~3.50]{Putz}.  Six are pure septic fields and the seventh is an
exceptional septic field.  Thus the established frontier can be summarized in
one sentence:
\begin{quote}
\emph{a nontrivial primitive solution would have to pass through one of seven
explicit irreducible degree-seven fields.}
\end{quote}
This is a remarkably narrow frontier, but it is not yet a contradiction.
One must still connect each field to an arithmetic object on which rational
points or residual Galois representations can be exhausted.  No classical
shortcut is available: because the three exponents are distinct, the
equation admits no binomial factorization over a cyclotomic field, so the
Kummer-type descent that a repeated exponent provides does not exist.
Every step below therefore passes through the cover $\Phi$ of
Section~\ref{sec:descent}, and the exceptional field through a Frey family.

A note of Zeraoulia Rafik \cite{ZeraouliaRafik}, posted on 21 September
2026 and not peer reviewed, studies the same degree-seven cover from the
monodromy side: geometric monodromy $A_7$, arithmetic monodromy $S_7$
with sign field $\Q(\sqrt{-7})$, a Frobenius-parity law at good primes,
and, conditionally on Putz's enumeration, the reduction of the branch
$7\mid Y$ to the exceptional field, with an exact certificate that this
field is the specialization $32/81$ of the cover.  Its splitting sieve at
$\ell=11,17,29$ leaves $36$ residue classes modulo $5423$; as the note
itself states, such a sieve necessarily retains the exceptional
specialization, and the note proves no residual irreducibility at $7$
and eliminates neither branch.  The contact with the present paper is in
the data rather than the method: its sign field $\Q(\sqrt{-7})$ is the
field of definition of the plane quartic of Section~\ref{sec:irreducible},
and its four surviving classes modulo $29$ are the classes
$\eta=10,14,24,28$ that the Hecke comparison at $29$ in
Section~\ref{sec:irreducible} eliminates.

\subsection{What is added here}

The principal new work is the elimination of those seven fields.

\begin{enumerate}[label=\textnormal{(\roman*)},leftmargin=2.2em]
\item For the six pure fields we construct the missing geometric bridge.
The thirty Fano planes attached to the septic cover split into two
$A_7$-orbits.  Their compatible pairs give a degree-$120$ resolvent whose
quotient is an explicit smooth plane quartic $C_+$ over
$\Q(\sqrt{-7})$.  A rational-parameter theorem shows that the relevant
parameter on $C_+$ takes rational values only at five points above
$0$, $1$, and $\infty$.  A nonzero primitive solution would have parameter
outside that set, so all six pure fields are excluded.

\item For the exceptional field we place the specialization in the Frey
family of Pacetti--Villagra Torcomian, use their modularity and conductor
theorems together with Hilbert level lowering, and reduce both the
irreducible and reducible residual alternatives to finite lists.  At a prime
above $29$, the four curve trace polynomials and the six possible modular or
character factors have $24$ nonzero resultants modulo $7$.  This excludes the
exceptional field.

\item We reconstruct the reducible-sector arguments of Dahmen--Siksek in an exact
certificate framework.  This is not presented as a new elimination of those
sectors.  It provides an independent route through the finite arithmetic.  The
identification of the quadratic factor from its intrinsic local labels,
without a field database, is Dahmen--Siksek's own argument
\cite[Proposition~4.2]{DahmenSiksek}, which we reconstruct as a certified
calculation; the one genuine refinement is that the ramified $7$-adic
step is carried out inside the formal neighbourhood of the given integral
Jacobian model.

\item We compose all four sectors into a single machine-checkable theorem
graph.  Every project-specific finite assertion is tied to exact input bytes,
an explicit predicate, and a replay program or authenticated completed log.
General theorems from the literature remain ordinary bibliographic dependencies.
\end{enumerate}

\subsection{A provenance-aware proof map}

The following table is the shortest guide to the logical ownership of the
proof.  ``Reconstructed'' means that this paper supplies an independent
exact certificate for a result whose mathematical argument was already in
the literature; it does not claim priority for that result.

\begin{center}
\small
\renewcommand{\arraystretch}{1.18}
\begin{tabular}{@{}>{\raggedright\arraybackslash}p{.19\textwidth}
                    >{\raggedright\arraybackslash}p{.25\textwidth}
                    >{\raggedright\arraybackslash}p{.46\textwidth}@{}}
\toprule
stage & status before this paper & role in this paper\\
\midrule
signed local descent and four factor sectors
 & Dahmen--Siksek
 & imported, restated, and independently checked\\
sectors $[1,6]$, $[2,5]$, $[3,4]$
 & eliminated by Dahmen--Siksek
 & reconstructed with certificates; an intrinsic formal-group refinement is added\\
irreducible sector $[7]$
 & reduced by Putz to seven explicit fields
 & the seven-field list is imported and authenticated\\
six pure septic fields
 & not eliminated by the cited prior work
 & excluded here by the Fano resolvent and the rational-parameter theorem\\
one exceptional septic field
 & modular tools available, but no terminal comparison
 & excluded here by level lowering, ray characters, and the $q=29$ resultant test\\
global theorem
 & one finite irreducible frontier remained
 & all four sectors are composed into the contradiction of Theorem~\ref{thm:main}\\
\bottomrule
\end{tabular}
\normalsize
\end{center}

\subsection{How the proof was completed}

The working principle was to treat a hypothetical solution not as a large
integer triple to be chased, but as one mathematical state seen through
several exact representations.  The same state has a descent algebra, local
factor labels, ideal classes, divisor classes, finite-quotient coordinates,
and $p$-adic analytic coordinates.  Each representation supplies a
constraint.  The proof succeeds when the intersection of all compatible
states is empty.

This viewpoint led to the following order of work:
\begin{enumerate}[leftmargin=2.2em]
\item fix the strongest established frontier before beginning a new search;
\item translate each remaining assertion into a finite interface with exact
inputs and a stated acceptance predicate;
\item use exploratory computation only to discover a model, prime, or
quotient, and then replace discovery by a deterministic certificate;
\item keep the four descent sectors separate until each has a terminal
statement, and only then compose them;
\item distinguish a recomputation of an existing argument from a genuinely
new bridge or obstruction.
\end{enumerate}
This is why the proof is modular.  A specialist can inspect the new
irreducible argument without repeating the reducible sectors, or can replay
one finite computation without trusting the remainder of the software tree.

The central object throughout is the polynomial
\begin{equation}\label{eq:phi}
 \Phi(T)=15T^7-35T^6+21T^5,
 \qquad \eta=\frac{X^5}{Z^7},
 \qquad f_\eta(T)=\Phi(T)-\eta.
\end{equation}
The associated \'etale algebra
\[
                         A_\eta=\Q[T]/(f_\eta)
\]
encodes a hypothetical solution in several complementary forms.  Local
factorization describes which global fields may occur.  Descent maps those
fields to curves.  Mordell--Weil information restricts the resulting points
to finitely many congruence classes, and local analytic functions decide
whether those classes contain a point compatible with the original rational
parameter.

The argument uses standard results in descent, number-field
enumeration, modularity, level lowering, formal groups, Mordell--Weil sieves,
and Chabauty methods.  The large finite calculations specific to
\eqref{eq:original} are exact and are separated into independently
checkable certificates.  Thus the proof has two visible layers: general
theorems, cited in the usual way, and finite arithmetic, supplied with the
paper.  Sections~\ref{sec:descent}--\ref{sec:reducible-end} reconstruct the
descent and reducible sectors.  Section~\ref{sec:irreducible} begins at
Putz's seven-field frontier and contains the new terminal argument.  The
final sections compose the proof and describe the certificates.

\section{The degree-seven descent algebra}\label{sec:descent}

\paragraph{Provenance.}
The local covering and exhaustive four-sector reduction in this section are
inputs from Dahmen--Siksek.  We record them because they are the
interface between a primitive solution and every later computation.

The derivative and the two branch values are given by
\begin{equation}\label{eq:derivative}
          \Phi'(T)=105T^4(T-1)^2,
          \qquad \Phi(0)=0,\quad \Phi(1)=1.
\end{equation}
Equivalently, $\Phi(T)=105\int_0^T u^4(u-1)^2\,du$.  With the
normalization $\Phi(0)=0$, $\Phi(1)=1$, it is the unique polynomial of
degree seven whose finite critical points are $0$, of multiplicity four,
and $1$, of multiplicity two: one point of ramification index $5$ above
$0$, one of index $3$ above $1$, total ramification above $\infty$, and
no other ramification.
In particular, $\Phi$ is strictly increasing on $\mathbb R$.  For a
nonzero solution, $\eta\notin\{0,1\}$, and $f_\eta$ is separable with one
real root.  A resultant computation gives
\begin{equation}\label{eq:discriminant}
 \disc(f_\eta)
   =-3^6 5^6 7^7\eta^4(\eta-1)^2.
\end{equation}
Its squareclass is therefore $-7$.

The exponents enter the argument only through the parameter.  Let
$\ell\notin\{3,5,7\}$ be prime.  Primitivity gives
\[
 \ell\mid X\ \Rightarrow\ 5\mid\ord_\ell(\eta),\qquad
 \ell\mid Y\ \Rightarrow\ 3\mid\ord_\ell(\eta-1),\qquad
 \ell\mid Z\ \Rightarrow\ 7\mid\ord_\ell(\eta),
\]
and $\eta$, $\eta-1$ are $\ell$-adic units when $\ell\nmid XYZ$.  Scaling
by the corresponding power of $\ell$ reduces the root cluster to a separable
binomial: $21T'^5-c$ with $T=\ell^{\ord_\ell(X)}T'$, or $35S'^3+c$ with
$T-1=\ell^{\ord_\ell(Y)}S'$, or $15W^7-X^5$ with $W=ZT$, where $c$ is an
$\ell$-adic unit; and \eqref{eq:discriminant} shows that all remaining
differences of roots are $\ell$-adic units.  Hence $A_\eta$ is unramified
outside $\{3,5,7\}$ \cite[Lemma~4.1(ii)]{DahmenSiksek}.  This is the
descent of Darmon and Granville \cite{DarmonGranville}: a hypothetical
solution of signature $(5,3,7)$ is detected only through these three
divisibilities.

The local covering theorem at $7$ in \cite[Proposition~3.5]{DahmenSiksek}
gives local factor-degree patterns $[1,6]$, $[2,5]$, $[3,4]$, and $[7]$.
Global factors induce groupings of the irreducible local factors, so their
degrees are obtained by coarsening these patterns.  Consequently the
possible global factor-degree patterns are exactly among
\begin{equation}\label{eq:four-sectors}
                  [1,6],\qquad [2,5],\qquad [3,4],\qquad [7].
\end{equation}
In particular, there are at most two global factors.  For example, if
$7\mid X$ and $r=\ord_7(X)$, the Newton polygon has vertices
$(0,5r),(5,1),(7,0)$, giving irreducible local factors of degrees $5$ and $2$.
If $7\mid Y$ and $r=\ord_7(Y)$, translation by $1$ gives vertices
$(0,3r),(3,1),(7,0)$, giving irreducible local factors of degrees $3$ and $4$.
The unit case, and the $7\mid Z$ case after scaling to an integral model,
give the remaining local patterns by the cited proposition.

We now eliminate the four sectors one at a time.

\section{The rational-factor sector}

\paragraph{Provenance.}
The mathematical elimination in this section is due to Dahmen--Siksek.  This
section gives a certificate-level reconstruction of their
Proposition~6.1 and Lemmas~6.2--6.8 \cite[Section~6]{DahmenSiksek}.

Assume that $r\in\Q$ is a root of $f_\eta$, and put $w=rZ$.
Clearing denominators gives
\begin{equation}\label{eq:rational-root-cleared}
       15w^7-35Zw^6+21Z^2w^5-X^5=0.
\end{equation}
The translated identity
\[
 \Phi(T)-1=(T-1)^3(15T^4+10T^3+6T^2+3T+1)
\]
also gives
\begin{equation}\label{eq:quartic-cube}
 (w-Z)^3
 (Z^4+3Z^3w+6Z^2w^2+10Zw^3+15w^4)=-Y^3.
\end{equation}

\subsection{Local restrictions and quadratic ideal descent}

Newton polygons at $3$, $5$, and $7$ give
\begin{equation}\label{eq:rational-local}
 3\nmid Z,\qquad 5\nmid YZ,\qquad 7\nmid Y,\qquad
 \ord_3(r)\in\{0,1,-1\},\qquad \ord_5(r)\geq0.
\end{equation}
The rational-root theorem applied to \eqref{eq:rational-root-cleared}
then shows that $w$ is an integer or one third of an integer.  Moreover, for
every prime $p$, at least one of $\ord_p(Z)$ and $\ord_p(w)$ is zero.

For all three reducible sectors fix the notation
\begin{equation}\label{eq:k-minus35}
 k=\Q(\omega)=\Q(\sqrt{-35}),\qquad
 \omega^2+\omega+9=0,\qquad \delta=2\omega+1=\sqrt{-35}.
\end{equation}
Thus $\OO_k=\Z[\omega]$.  Put
\[
 A=Z^2+wZ-3w^2,\qquad B=-wZ-w^2,\qquad \mathcal F=A+B\omega.
\]
Then
\[
 \mathcal F\overline{\mathcal F}=A^2-AB+9B^2
        =Z^4+3Z^3w+6Z^2w^2+10Zw^3+15w^4.
\]
The class group of $k$ is $C_2$.  The prime $3$ splits, and exact ideal
arithmetic reduces \eqref{eq:quartic-cube} to three possibilities:
\begin{equation}\label{eq:three-cubic-cases}
 \mathcal F=\alpha(u+v\omega)^3,\qquad
               \alpha\in\{1,9\omega,3\omega^2\},\qquad u,v\in\Q.
\end{equation}
For $\alpha=1$, \cite[Lemma~6.4]{DahmenSiksek} further gives
$w,u,v\in\Z$, with $\gcd(Z,w)=\gcd(u,v)=1$.

Expanding \eqref{eq:three-cubic-cases} produces three genus-two curves.
One has no $\Q_2$-point because its primitive homogeneous sextic is always
$5$ modulo $8$.  A second has precisely the two rational points
$(-2,\pm540)$; its Jacobian has rank one and the rational points are
exhausted by Chabauty's method.  Substitution gives $r=-1/4$ and
\[
                         \Phi(-1/4)=-\frac{491}{2^{14}},
\]
which cannot be a reduced quotient of a fifth power by a seventh power.

\subsection{The main genus-two curve}

The remaining case gives the curve
\begin{equation}\label{eq:rational-genus2}
 C_1:\quad
 y^2=x^6-6x^5+69x^4-38x^3+6x^2-240x+1057.
\end{equation}
Its Jacobian is torsion-free of rank three, and three explicit Mumford
divisors form a $\Z$-basis.  A Mordell--Weil sieve in the sense of
\cite{BruinStollSieve}, modulo $1080J_1(\Q)$,
places every rational point in one of the six classes represented by
\[
 \infty_+,\ \infty_-,\ (1/2,\pm245/8),\ (3,\pm65).
\]
Exact formal-group calculations on these classes imply, for every affine
$(x,y)\in C_1(\Q)$,
\begin{equation}\label{eq:rational-local-alternative}
                         \ord_3(x)\leq-3
                    \quad\hbox{or}\quad
                         \ord_5(y)\geq1.
\end{equation}

The point on \eqref{eq:rational-genus2} arising from
\eqref{eq:three-cubic-cases} has
\[
 x=\frac uv,\qquad
 y=\frac{(Z+w)^2+3w^2}{v^3}.
\]
The case $v=0$ is also visible on the projective model.  Coprimality forces
$u=\pm1$, while the two descent equations give
\[
       Z^2+wZ-3w^2=\pm1,\qquad -w(Z+w)=0.
\]
If $w=0$, then $X=0$; if $Z=-w$, then $-3w^2=\pm1$.  Both are impossible
under the hypotheses.

For $v\neq0$, the second arm of
\eqref{eq:rational-local-alternative} would imply
$5\mid (Z+w)^2+3w^2$, impossible for a primitive pair modulo $5$.
Hence $27\mid v$.  The second descent equation gives
$27\mid w(Z+w)$.  Coprimality and \eqref{eq:rational-local} exclude
$27\mid w$, so $27\mid Z+w$.  Reducing the first descent equation modulo
$3$ now gives $3\mid u$, contradicting $\gcd(u,v)=1$.

\begin{proposition}\label{prop:rational-sector}
The polynomial $f_\eta$ has no rational root.  Hence the sector $[1,6]$
is empty.
\end{proposition}

\section{The quadratic-factor sector}

\paragraph{Provenance.}
The reduction to the fibre product and its six rational points is
Dahmen--Siksek, Proposition~5.1 and Lemma~5.2
\cite[Section~5]{DahmenSiksek}.  We include the construction because it is
the complete bridge from a quadratic root to the final fifth-power test.  The
database-free recovery of $\Q(\sqrt{-35})$ from its local labels is also
Dahmen--Siksek's, Proposition~4.2 of their paper (imaginary, unramified
outside $\{3,5,7\}$, inert at $2$, ramified at $7$); we reconstruct it
here as a certified calculation, and an earlier draft's description of
it as a new refinement was an error.

Suppose that $A_\eta$ has a quadratic direct factor $L$.  As in
\cite[Proposition~4.2]{DahmenSiksek}, its real signature, ramification,
and dyadic labels identify it without a field database.

Since $f_\eta$ has exactly one real root, $L$ is imaginary.  At every prime
outside $\{3,5,7\}$, the scaled local models are separable and every direct
factor is unramified.  At $2$, the possible local degree patterns are
$[1,2,4]$ and $[1,3,3]$.  A quadratic global direct factor must therefore
be the unique unramified quadratic component in the first pattern.  Thus
$2$ is inert in $L$.  The negative odd fundamental discriminants supported
on $\{3,5,7\}$ and inert at $2$ are $-3$ and $-35$.  Finally, the
degree-two Newton-polygon component at $7$ is ramified.  Therefore
\begin{equation}\label{eq:quadratic-field}
                              L=k=\Q(\omega).
\end{equation}

Let $u$ be a nonrational root of $f_\eta$ in $L$ and $v=\bar u$.  Symmetric
elimination gives a rational point on
\begin{equation}\label{eq:fibre-product}
 D:\quad
 \begin{cases}
 \beta^2=\alpha^4-4\alpha^3+6\alpha^2+9,\\
 \alpha(\alpha-4)=-35\gamma^2.
 \end{cases}
\end{equation}
Here
\[
                    \alpha=\frac{(u+v)^2}{uv}.
\]

\begin{proposition}\label{prop:six-points}
The rational points on $D$ are exactly
\[
 (0,\pm3,0),\qquad
 \left(\frac{35}{9},\ \pm\frac{782}{81},\ \pm\frac19\right),
\]
with independent signs in the second family.
\end{proposition}

\begin{proof}
Map $D$ to the genus-two curve
\begin{equation}\label{eq:quadratic-genus2}
 -35y^2=x(x-4)(x^4-4x^3+6x^2+9).
\end{equation}
The \'etale algebra of the right side is
$\Q\times\Q\times L_4$, where
\[
 L_4=\Q(\vartheta),\qquad
 \vartheta^4-4\vartheta^3+6\vartheta^2+9=0.
\]
An exact $2$-cover descent \cite{BruinStollDescent} gives three
fake-Selmer classes, whose components
at $0$, $4$, and $\vartheta$ are
\[
\begin{array}{c|ccc}
 &0&4&\vartheta\\ \hline
 \xi_1&-3&1&3(\vartheta-4)\\
 \xi_2&-35&1&1\\
 \xi_3&-35&1&-5.
\end{array}
\]
The first class cannot lift to a nonendpoint rational point on $D$.
Indeed, in this class $x=-3ta^2$ and $x-4=tb^2$, whereas the additional
condition on $D$ requires $x(x-4)/(-35)$ to be a square.  This would make
$3/35$ a square, which is impossible already over $\Q_7$.  The endpoint
$x=4$ does not lift either, since it would require $\beta^2=105$.
The other two classes map to elliptic curves over $L_4$.

Their
ranks are respectively $3$ and $2$, their torsion is trivial, and elliptic
Chabauty \cite{BruinEC} retains only the rational $x$-coordinates
\[
                               x=0,\qquad x=35/9.
\]
For the rank-three curve, completeness is certified by saturation at every
prime dividing the Chabauty index
\[
\begin{split}
6592880678883323322765152793600
={}&2^{10}3^5 5^2 7^4\cdot11\cdot17\cdot19\cdot29\cdot31\cdot37\\
 &\cdot43\cdot47\cdot53\cdot61\cdot71\cdot83\cdot97.
\end{split}
\]
Substitution into \eqref{eq:fibre-product} gives precisely the six points
stated.
\end{proof}

If $\alpha=0$, then $v/u=-1$.  The root relation gives either $u=0$ or
$u=\pm\delta/5$.  The nonzero alternatives satisfy
\[
                            \Phi(u)=\frac{7^4}{5^2},
\]
which is not a reduced quotient of a fifth power by a seventh power.  If
$\alpha=35/9$, then
\[
 \frac vu=\frac{17\pm\delta}{18}.
\]
The root relation gives either a branch value or the rational number
\[
 \frac{495535757310206240323260222}
      {432344158742282808792052775},
\]
whose coprime numerator and denominator factor as
\[
\begin{aligned}
495535757310206240323260222
 &=2\,3^{16}7^4\,17\,19^{10}\,23,\\
432344158742282808792052775
 &=5^2\,29^7\,139^7.
\end{aligned}
\]
The numerator is not a fifth power.  This proves:

\begin{proposition}\label{prop:quadratic-sector}
The sector $[2,5]$ is empty.
\end{proposition}

\section{The cubic--quartic sector}\label{sec:reducible-end}

\paragraph{Provenance.}
The elimination follows Dahmen--Siksek, Section~7, in particular
Proposition~7.1 and Lemmas~7.2--7.7 \cite[Section~7]{DahmenSiksek}.
Our additional contribution is to certify the ramified local step inside the
intrinsic integral formal neighbourhood, rather than treating model-dependent
coordinates as an unexamined proxy for the formal group.

The quartic direct factor is totally imaginary and unramified outside
$\{3,5,7\}$.  Dahmen--Siksek consider the complete lists of nineteen cubic
fields of signature $(1,1)$ and seventeen quartic fields of signature
$(0,2)$ unramified outside this set.  They compare each cubic--quartic
product with the permitted local algebras at $3$, $5$, and $7$.
Their Proposition~4.3 leaves two products, whose cubic factors are
$\Q(\sqrt[3]{5^\alpha7})$ for $\alpha=1,2$, and whose quartic factor is the
same field in both cases \cite[Proposition~4.3]{DahmenSiksek}.

The surviving field is
\begin{equation}\label{eq:quartic-field}
 M=\Q[a]/(a^4-a^3+3a^2+3a+9)
  =k\!\left(\sqrt{\frac{7+\delta}{2}}\right).
\end{equation}
The associated quotient is the genus-two curve
\begin{equation}\label{eq:section7-curve}
 C_2:\quad V^2=U(U-4)(U^4-4U^3+6U^2+9)
\end{equation}
over $k$.  Write $J_2=\operatorname{Jac}(C_2)$.

Here is the solution-bearing condition used in the terminal argument.
Put $\kappa=(7+\delta)/2$, choose a root $u$ generating
$M=k(\sqrt\kappa)$, and let $v$ be its conjugate over $k$.
The quotient map of \cite[Lemma~7.2]{DahmenSiksek} gives a point
$P\in C_2(k)$ with
\[
 U(P)=\frac{(u+v)^2}{uv},\qquad \rho(P)=\Phi(u)=\eta\in\Q,
\]
where $\rho\in k(C_2)$ is the quotient parameter.  The denominator
$uv$ is nonzero because $\eta\ne0$.  Moreover $U(P)\notin\Q$.
Indeed, $N_{k/\Q}(\kappa)=21$ is not a square in $k$, so $M/\Q$
is not Galois and its only quadratic subfield is $k$.
If $U(P)\in\Q$, the equation
$(u/v)^2+(2-U(P))(u/v)+1=0$ implies $u/v\in k$.
Since $uv\in k$, we obtain $u^2\in k$, and hence
$u=b\sqrt\kappa$ for some nonzero $b\in k$.
The coefficient of $\sqrt\kappa$ in $\Phi(u)\in k$ then gives
$b^2=-7/(5\kappa)$.  This is impossible: its norm would give
$N_{k/\Q}(b)^2=7/75$, which is not a rational square.

\subsection{The Mordell--Weil group and sieve}

We use the Mordell--Weil group determination of
\cite[Lemma~7.3]{DahmenSiksek}, which gives
\begin{equation}\label{eq:section7-MW}
 J_2(k)=\langle D_0\rangle_{10}
       \oplus\Z D_1\oplus\Z D_2\oplus\Z D_3.
\end{equation}
For reproducibility, we also give a direct finite-field certificate for the
remaining $2$-saturation in this statement.  Put
$A=J_2(k)$, $H=\langle D_0,D_1,D_2,D_3\rangle$, and
$H_0=\langle D_0,D_1,D_2,D'_3\rangle$, where $D'_3$ is the class
introduced in the proof of \cite[Lemma~7.3]{DahmenSiksek}.  The rational and twisted
Mordell--Weil computations used in the cited lemma give
$\rank A=3$; good reduction at $11$ and $13$, together with the displayed
order of $D_0$, gives $A_{\rm tors}=\langle D_0\rangle\simeq\Z/10\Z$.
For the nontrivial $\sigma\in\Gal(k/\Q)$ and every $P\in A$,
\[
                  2P=(P+\sigma P)+(P-\sigma P).
\]
The full rational and twisted groups identify the fixed and anti-fixed
summands on the right with the named generators of $H_0$.
Thus $2A\subseteq H_0\subseteq H$, the second inclusion following from
the verified relation
\[
                 D'_3=9D_0+3D_1-D_2-2D_3.
\]
It therefore remains only to show that the four displayed classes span
$A/2A$.

At the good prime $\mathfrak l=(173,\delta-22)$, the change of variables
\[
 u=-36/U,\qquad v=1296V/U^3
\]
puts the curve into the odd monic form
\[
 v^2=u^5+9u^4+864u^3+28512u^2+373248u+1679616.
\]
Its five roots modulo $\mathfrak l$, in the certified order, are
$15,90,119,122,164$.  Applying the usual odd-degree $u-T$ Kummer map and
then the five quadratic characters of $\F_{173}^{\times}$ to
$D_0,D_1,D_2,D_3$ gives the matrix
\[
 \begin{pmatrix}
 0&0&1&1\\
 1&0&1&1\\
 1&0&1&0\\
 1&1&0&1\\
 1&1&1&1
 \end{pmatrix}.
\]
The determinant of its first four rows is $1$ over $\F_2$.  Hence these
four classes are independent in $A/2A$.  Since
$\dim_{\F_2}A/2A=3+1=4$, they form a basis; as $A/H$ was already killed by
$2$, this proves $A=H$.  The archive contains both the short reconstruction
and an independent evaluation of every Kummer entry.

A Mordell--Weil sieve \cite{BruinStollSieve} with modulus $2240$ yields
twelve points
$P_1,\ldots,P_{12}$ such that every $Q\in C_2(k)$ satisfies
\begin{equation}\label{eq:section7-sieve}
            [Q-P_i]\in2240J_2(k)
\end{equation}
for a unique $i$.

Let $\PP$ be the unique prime of $k$ above $7$, let
$k_\PP$ be the completion, and put $R=\OO_{k_\PP}$, $\m=(\delta)$.  Thus
\[
                         v_\PP(\delta)=1,\qquad v_\PP(7)=2.
\]
Use Flynn's integral $\mathbf P^{15}$ model of $J_2$ with origin
$(1:0:\cdots:0)$ and normalized coordinates $s_i=a_i/a_0$.  On the
origin neighbourhood $\mathcal N$, the projective-coordinate framework of
\cite{CasselsFlynn} and Corollary~2.2 and Theorem~3.5 of \cite{Flynn}
identify $(s_1,s_2)$ as integral formal parameters and give an
integral formal group law.  Define
\begin{equation}\label{eq:G2}
 G[2]=\{A\in\mathcal N:s_1(A),s_2(A)\in\m^2\}.
\end{equation}
It is a subgroup, and the coordinate identities place every nonorigin
normalized coordinate in $\m^2$.

The exact origin-chart valuations of $2240D_1$, $2240D_2$, and $2240D_3$
in the adapted integral coordinates described in Appendix~\ref{app:flynn}
are
\begin{equation}\label{eq:section7-valuations}
\begin{array}{c|l}
D_1&0,2,2,4,4,4,6,6,6,6,8,8,8,8,8,16\\
D_2&0,5,3,10,8,6,15,13,11,9,20,18,16,14,12,20\\
D_3&0,2,2,4,4,4,6,6,6,6,8,8,8,8,8,16.
\end{array}
\end{equation}
The torsion generator is killed by $2240$, so
\begin{equation}\label{eq:2240G2}
                            2240J_2(k)\subset G[2].
\end{equation}

\subsection{The twelve local classes}\label{subsec:twelve-local}

For each $P_i$, exact pullback calculations prove
\begin{equation}\label{eq:local-pullback}
 [Q-P_i]\in G[2]\quad\Longrightarrow\quad
                         Q\equiv P_i\pmod{\PP^2}.
\end{equation}
At each of five regular finite bases, the $98$ depth-two congruence classes
split into one base class and $97$ rejected classes.  At each of the five
ramified finite bases, the complete tree is
\[
\begin{array}{c|rrrr|rr}
n&\text{tested}&\text{base}&\text{rejected}&\text{unresolved}
 &49\,\text{-digit lifts}&\text{admissible at }n+1\\ \hline
2&98 &1&91 &6  &294 &294\\
3&294&0&252&42 &2058&686\\
4&686&0&588&98 &4802&686\\
5&686&0&686&0  &--&--.
\end{array}
\]
The last two columns make the branching arithmetic explicit.  Each
unresolved pair modulo $\PP^n$ has $7^2=49$ two-coordinate digit lifts;
the next row counts only those lifts that also satisfy the curve equation
modulo $\PP^{n+1}$.  At the ramified special-fibre point this number is not
uniformly seven, which is why the tested-node column is not obtained by
multiplying the unresolved column by a fixed branching factor.
The singular special-fibre point $U\equiv4\pmod{\PP}$ is included in this
calculation.  At the two points at infinity, the opposite branch lies
outside $\mathcal N$ and the identity branch has
\[
                            s_1=\pm\tfrac12w+O(w^2),
\]
with unit linear coefficient.  Here a node is a full congruence class, and a
rejection is a uniform valuation statement on that class; no finite sample
is used as an exhaustion argument.

Combining \eqref{eq:section7-sieve}, \eqref{eq:2240G2}, and
\eqref{eq:local-pullback}, every $k$-point lies in one of the twelve
explicit $\PP^2$-disks.  Exact $7$-adic and $23$-adic conditions, followed
where necessary by Coleman integration, show that no disk contains a
solution-bearing point beyond its reference $P_i$.  The reference points
themselves either have parameter $0$, $1$, or $\infty$, or fail the required
rational admissibility condition; hence none yields a nonzero primitive
solution.

\begin{proposition}\label{prop:quartic-sector}
The sector $[3,4]$ is empty.
\end{proposition}

\section{The irreducible septic sector}\label{sec:irreducible}

\paragraph{Provenance.}
The seven-field enumeration is Putz's classification theorem.  The elimination of
the six pure fields and the exceptional field is the new terminal part of the
argument.

Suppose now that $A_\eta$ is a field.  It has degree $7$, signature
$(1,3)$, is unramified outside $\{3,5,7\}$, and has the prescribed
local algebras at $2$, $3$, $5$, and $7$ of
\cite[Corollary~3.4 and Propositions~3.5--3.7]{DahmenSiksek}.
Putz's exhaustive Hunter enumeration
\cite[Theorem~3.50]{Putz} gives exactly seven possibilities:
\begin{equation}\label{eq:seven-fields}
 \Q\!\left(\sqrt[7]{3^a5}\right)\quad(1\leq a\leq6),
 \qquad
 \Q[T]/(T^7-483T^2+3955T-3945).
\end{equation}
We call the first six fields pure and the last field exceptional.

\subsection{The pure fields}

The Fano-resolvent construction and the rational-parameter theorem in this
subsection supply the geometric obstruction that was absent from the cited prior work.

The normal closure of a pure field has Frobenius group $F_{42}$ and unique
quadratic subfield
\[
                              K_7=\Q(s),\qquad s^2=-7.
\]

\subsubsection{The Fano-resolvent bridge}

The passage from a pure septic field to a curve is as follows.  Scale
$V=15T$ and write
\begin{equation}\label{eq:scaled-septic}
 h_t(V)=V^7-35V^6+315V^5-15^6t,
 \qquad t=\eta=\frac{X^5}{Z^7}.
\end{equation}
Let $v_1,\ldots,v_7$ be its roots.  A labelled Fano plane $F$ on these
seven roots has seven lines, each a three-element subset; define
\begin{equation}\label{eq:fano-invariant}
 M_F=\sum_{\{i,j,k\}\text{ a line of }F}v_iv_jv_k,
 \qquad q_F=\frac{M_F}{225}.
\end{equation}
There are thirty labelled Fano planes.  They form two $A_7$-orbits
$\mathcal F_+$ and $\mathcal F_-$ of fifteen planes.  Since
\[
 \disc(h_t)=-7^7 15^{36}t^4(t-1)^2,
\]
the two orbit polynomials are conjugate over $K_7$.  After fixing the
Vandermonde orientation, put
\begin{equation}\label{eq:plus-resolvent}
 R_+(q,t)=\prod_{F\in\mathcal F_+}(q-q_F)
          =\frac{U(q,t)+sW(q,t)}2,
 \qquad U,W\in\Z[q,t].
\end{equation}
The normalization of $R_+=0$ is denoted $C_+$.  On its fifteen sheets the
branch permutations above $t=0,1,\infty$ have cycle structures
\[
                         5^3,\qquad3^5,\qquad7^2 1.
\]
Riemann--Hurwitz therefore gives $g(C_+)=3$.

To see why this quotient detects the six pure fields, call
$F\in\mathcal F_+$ and $G\in\mathcal F_-$ compatible when they have no
line in common.  We use unordered pairs $\{F,G\}$, since odd
permutations exchange the two $A_7$-orbits.  Each $F$ has exactly eight
compatible $G$'s, giving 120 such pairs.  The normalization of the
corresponding cover is a curve $X_{42}/\Q$.  The stabilizer of an
unordered compatible pair in $A_7$ is the Frobenius group $F_{21}$,
and its stabilizer in $S_7$ is $F_{42}$.  Thus $X_{42}$ has degree $120$
over the $t$-line and, after base change to $K_7$, the forgetful map
\begin{equation}\label{eq:eight-over-fifteen}
                         X_{42,K_7}\longrightarrow C_+
\end{equation}
has degree $8$.  The branch permutations on the 120 sheets have cycle
structures $5^{24}$, $3^{40}$, and $7^{17}1$, so $g(X_{42})=20$.

If \eqref{eq:scaled-septic} cuts out one of the pure fields
$\Q(\sqrt[7]{3^a5})$, its normal closure has group $F_{42}$.  The
fixed-field/resolvent correspondence supplies a rational point on
$X_{42}$, hence a $K_7$-point on $C_+$ through
\eqref{eq:eight-over-fifteen}.  Both maps preserve $t$.  A nonzero
solution has $t\notin\{0,1,\infty\}$, so the resulting point is not in a
branch fibre.  This constructs the claimed bridge rather than merely
inferring it from matching Galois groups.

The canonical model of $C_+$ is the explicit smooth plane quartic printed
in Appendix~\ref{app:Cplus}.  Its smoothness proves again that its genus is
$3$.  The inverse parameter is the exact rational function
$t=-G_t/H_t$; all coefficients of the two quintics are printed in
Table~\ref{tab:t-parameter-coefficients}.
The archive file \path{p7_f42_canonical_exact.json.gz} specifies
the forward coordinates $N_j(q,t)$ for $j=1,2,3$, and the inverse
forms $G_q,H_q$.  Exact reduction modulo $R_+(q,t)$ verifies
\[
 Q(N_1,N_2,N_3)=0,\qquad
 G_q(N)+qH_q(N)=0,\qquad G_t(N)+tH_t(N)=0,
\]
where $Q$ is the printed quartic and the two inverse denominators are
nonzero in the function field.  Recovering both generators $q$ and $t$
proves that this map induces an isomorphism of function fields, hence of
the smooth projective models, preserving the parameter.

\begin{theorem}[rational-parameter theorem]\label{thm:rational-t}
Let $J_+=\operatorname{Jac}(C_+)$.  Then
\[
 \{Q\in C_+(K_7):t(Q)\in\Pone(\Q)\}
       =\{P_1,P_2,P_3,P_4,P_5\}.
\]
The corresponding parameter values are $0,1,\infty,\infty,\infty$.
\end{theorem}

\begin{proof}
Put $D_i=[P_i-P_1]$ and $H_0=\langle D_2,D_3,D_4,D_5\rangle$.
Appendix~\ref{app:rank-proof} gives the complete containing-space and local
stopping argument in the genus-three framework of \cite{BPS}, including
the $7$-adic kernel witness and the literal $5$-adic obstruction.  Together
with the independently checked lower bound, it proves
\[
 \dim_{\F_2}\operatorname{Sel}_2(J_+/K_7)=4,
 \qquad \rank J_+(K_7)=4,
 \qquad \Sha(J_+/K_7)[2]=0.
\]
The independent $2$-saturation witness shows that $D_2,D_3,D_4,D_5$
form a basis modulo $2$.  Thus $H_0$ has odd index.  The existence of the fifth-division class used below can
be seen from the branch divisor.  Write $A=(t)_0/5$, a $K_7$-rational
divisor of degree three, and set $B=[A-3P_1]$.  The pole divisor is
$(t)_\infty=7P_3+7P_4+P_5$, so
\[
 5B=7D_3+7D_4+D_5.
\]
Consequently the class $E=3B-4D_3-4D_4$ satisfies
\begin{equation}\label{eq:Erelation}
                         5E=D_3+D_4+3D_5.
\end{equation}
Set $H_1=\langle D_2,E,D_4,D_5\rangle$.  Equation~\eqref{eq:Erelation}
gives $H_0\subseteq H_1$.  Four independent local $C_{25}$ rows have
determinant $4$ modulo $5$, and $J_+(K_7)[5]=0$; hence $H_1$ is
$5$-saturated.  Its index is therefore coprime to $400$, and
$H_1\to J_+(K_7)/400J_+(K_7)$ is surjective.  This is the coefficient
coverage used in the sieve.

At the split primes
\[
                     23,43,67,71,79,107,109,113,
\]
both reductions of the curve are enumerated, as are the images of every
reduced point in the relevant local quotients of $J_+$.  The same global
coefficient vector must occur at both places.  In addition, rationality of
$t(Q)$ imposes equality of the two projective reductions of $t$.  The
resulting finite sieve has the following exact counts:
\[
\begin{array}{lr}
\toprule
\text{stage}&\text{surviving coefficient classes}\\
\midrule
\text{coupled $2$-primary sieve modulo $16$}&592\\
\text{$5$-primary prefilter modulo $25$}&1252\\
\text{joint point fingerprints modulo $400$}&144\\
\bottomrule
\end{array}
\]
The final classes use exactly the five vectors modulo $25$
\[
 (0,0,0,0),\ (1,0,0,0),\ (0,5,24,22),\
 (0,0,1,0),\ (0,0,0,1),
\]
and their projection at the two primes above $23$ consists precisely of the
five ordered residue pairs of $P_1,\ldots,P_5$.

The combined $23$-adic logarithm matrix has size $6\times4$ and rank $4$.
The resulting two independent logarithmic functionals vanish on $H_0$.
Since $H_0$ has finite index, they vanish on all of $J_+(K_7)$: for any
$P\in J_+(K_7)$, a nonzero integer multiple of $P$ lies in $H_0$, and
the logarithm takes values in characteristic zero.  At the disks of $P_2,P_3,P_4,P_5$, the Siksek tangent
determinants modulo $23$ are
\begin{equation}\label{eq:siksek-dets}
                              12,\quad16,\quad4,\quad19.
\end{equation}
This is the number-field Chabauty criterion of \cite{Siksek}; the values are
units, so each disk contains only its reference point among common
Coleman zeros.

At $P_1$, write the two local parameters as $23u,23v$, so that the
whole product residue disk is parametrized by $(u,v)\in\Z_{23}^2$.
For the two annihilator integrals $F_1,F_2$, use the normalized equations
$F_1/23=0$ and $(F_2-6F_1)/23^2=0$.  The second numerator is integrally
divisible by $23^2$: its linear coefficients vanish modulo $23$ before
scaling, and all higher terms supply the extra factor.  The reductions are
\[
              2u+4v=0,\qquad
              21u+6v+14u^2+21v^2=0.
\]
They have exactly two simple reduced zeros,
\[
                         (u,v)=(0,0),\qquad(14,16),
\]
with determinants $20$ and $3$.  Multivariate Hensel lifting gives one
exact common zero in each of these two classes and none elsewhere.
The first is $P_1$.  The arithmetic retains seven digits after both
normalizing divisions: an omitted integrated term of exponent $k\ge12$
has valuation at least $k-v_{23}(k)\ge12$, and the coefficient arithmetic
retains the required precision.  Multiplication by $23$ gives the local
coordinates modulo $23^8$.  The canonical parameter denominator is a
unit at both lifted endpoints, so its evaluation preserves those eight
digits.  At the second zero, the two parameter values modulo
$23^8=78{,}310{,}985{,}281$ are
\[
 t_{s=4}=59{,}233{,}664{,}629,\qquad
 t_{s=19}=24{,}239{,}267{,}738.
\]
Their difference has exact $23$-adic valuation $5$, with first nonzero
quotient digit $9$.  Hence this common Coleman zero does not have rational
$t$.  The five reference points are therefore the entire rational-parameter
locus.
\end{proof}

Each $P_i$ in Theorem~\ref{thm:rational-t} lies above one of the branch
values $0$, $1$, and $\infty$.  A nonzero primitive solution would give a
point with rational parameter outside this set.  Consequently none of the six pure fields occurs.

\subsection{The exceptional field}

The general modularity and conductor theorems used here are cited inputs.
The project-specific specialization, exhaustive packet and ray-character
interfaces, and the terminal $q=29$ comparison form the new finite
elimination.

We use the Frey abelian variety of Pacetti and Villagra Torcomian
\cite{PVT} for
\begin{equation}\label{eq:PVT-equation}
                             a^5+b^p+c^3=0.
\end{equation}
In our variables
\begin{equation}\label{eq:PVT-map}
                (a,b,c)=(X,-Z,Y),\qquad
                t_0=-\frac{a^5}{c^3}=\frac{\eta}{\eta-1}.
\end{equation}
Put $F=\Q(\sqrt5)$ and let
\(\bar\rho_7:G_F\to\operatorname{GL}_2(\overline{\F}_7)\) be the residual
representation of the weight-two realization of their plus motive at $t_0$.  The use of the modular
method at the small exponent $p=7$ rests on the following precise inputs,
not on the headline large-exponent theorem of \cite{PVT}.
Theorem~2.4(1) of \cite{PVT} states that for $r=3$ and $q\geq5$ the plus
motive is modular for every nonzero rational specialization; the
specializations used below are nonzero.  Its determinant is
the mod-$7$ cyclotomic character (equivalently, the compatible system has
trivial Nebentypus), and Theorem~2.1(3)--(4) gives the needed ramification
control and finiteness at every prime above $7$.  Here $7$ is inert and
unramified in $F$, so there is a single such prime.

Now assume that $\bar\rho_7$ is irreducible.  Apply
\cite[Theorem~7.8]{PVT} with $p=7$ to the specialization
\eqref{eq:PVT-map}.  This specialized theorem incorporates modularity,
the local conductor calculation, and Hilbert level lowering; its proof
cites \cite{BreuilDiamond} for the latter.  It gives a parallel-weight-two
Hilbert newform over $F$, with trivial Nebentypus,
whose residual representation is $\bar\rho_7$ and whose level is one of
\begin{equation}\label{eq:four-levels}
 3^{\epsilon_3}(\sqrt5)^{\epsilon_5},\qquad
 (\epsilon_3,\epsilon_5)\in\{2,3\}^2.
\end{equation}

The conclusion here is a newform conclusion, but the finite elimination below
does not depend on reconstructing a new quotient.  Let $V_N$ be the full
mod-$7$ cuspidal Brandt module at one of the four levels~$N$ in
\eqref{eq:four-levels}.  Our earlier computation first imposed eighteen
local trace-polynomial conditions, so its output must not be called the full
ambient module.  We instead verify the coverage of every condition.

Section~7.4 of \cite{PVT} gives an exhaustive trichotomy at an auxiliary
prime: ordinary specialization, one of the two degenerations $t=0,\infty$,
or the level-lowering specialization $t=1$.  From the complete upstream
\texttt{Data.txt}, the certificate independently reduces all $6253$ trace
polynomials at $37$ primes modulo~$7$.  Separately, it reconstructs all
$1264$ ordinary parameters at those $18$ filter primes and all
$128$ boundary character possibilities at the $16$ split primes.  It
performs the required relative-degree-two trace transport for the boundary
data and adjoins both signs of $N(\mathfrak q)+1$.  Seventeen of those eighteen filters are exactly
the resulting squarefree allowed-trace polynomials.  At $\ell=131$ the
degree-$48$ filter in our earlier computation omitted the allowed level-lowering trace $1$,
since $-(131+1)\equiv1\pmod7$.  We therefore replace only that condition, in
memory, by $T^{49}-T$, which admits every $\F_{49}$-valued trace.

Write $W_N\subseteq V_N$ for the intersection after these seventeen verified
conditions and the condition at $131$ imposing no further restriction
on $\F_{49}$-valued Frey eigenvalues.  Characteristic-zero Jacquet--Langlands
\cite[Theorem~3.9 and \S4]{DembeleVoight} supplies the corresponding
Hecke eigenvector in the definite quaternionic function model.  The exact
order and neighbour computations verify the maximal order, class number
one, and the complete Hecke operators; all finite stabilizer orders divide
$120$, hence are prime to $7$.  The integral function lattice and its
cuspidal mass condition are therefore preserved at primes above $7$.
Scale the eigenvector to be primitive in this local lattice and reduce.
The resulting nonzero vector has the required residual eigenvalues.
The local trichotomy places it in $W_N$.  This proves coverage without
degeneracy-map subtraction.  A fresh simultaneous primary decomposition has
packet dimensions, with multiplicity, summing to $\dim W_N$ at every level.
The exceptional field is unramified at $2$ with local factor degrees
$[1,2,4]$.  If $2\mid Z$, the integral equation for $W=ZT$ reduces to
\[
 W^7+1=(W+1)(W^3+W+1)(W^3+W^2+1)\quad\text{over }\F_2.
\]
Its factors are distinct and irreducible, giving local degrees $[1,3,3]$,
a contradiction.  Hence $Z$ is odd, and exactly one of $X,Y$ is even.
In the variable convention $X^5+Y^3=Z^7$ used here, the prime-$2$
calculation gives the necessary residual trace $-1$ when $Y$ is even
and $0$ when $X$ is even.  Retaining every packet for which that scalar
is a root of the residual $T_2$ polynomial gives
\[
\begin{array}{c|rrrr}
 &22&23&32&33\\ \hline
 \dim W_N&8&38&28&98\\
 \text{all packets}&4&14&12&24\\
 Y\text{-even}&1&2&2&5\\
 X\text{-even}&2&10&3&11.
\end{array}
\]
Across all thirty-six retained packet occurrences, the set of residual
$T_{\mathfrak q}$ annihilators at $\mathfrak q\mid29$ is exactly
\[
 T,\quad T^2,\quad T+2,\quad T^2+5T+2,\quad T^2+3T+4.
\]
The terminal list already contains all five, because $T+2=T-5$ modulo~$7$;
together with $T-2$ it is the same six-factor list used in the resultant
calculation below.  Thus every residual representation supplied by level
lowering survives the preceding filters and is covered before the terminal
comparison.  The old/new calculation remains only a diagnostic.  We make no
complete-newspace claim and require no mod-$7$ injectivity or semisimplicity
statement for degeneracy maps.

Suppose instead that $\bar\rho_7$ is reducible.  Its semisimplification has
the determinant-compatible form
\begin{equation}\label{eq:reducible-semisimplification}
 \bar\rho_7^{\mathrm{ss}}
       =\psi\oplus\psi^{-1}\bar\chi_7.
\end{equation}
At a finite prime $\lambda\nmid7$, the cyclotomic character is unramified,
so the two characters in \eqref{eq:reducible-semisimplification} have the
same Artin-conductor exponent.  The conductor of the semisimplification
is at most that of the representation.  Thus the local conductor bounds
at the primes above $3$ and $5$ give
\[
 2a_\lambda(\psi)
 =a_\lambda(\bar\rho_7^{\mathrm{ss}})
 \leq a_\lambda(\bar\rho_7)\leq3,
 \qquad a_\lambda(\psi)\leq1.
\]
At every other finite prime away from $7$, the character $\psi$ is
unramified.

The prime $\mathfrak p=(7)$ is unramified and inert in $F$, with residue
field $\F_{49}$.  By the finite-flat character classification
\cite[Corollary~3.4.4]{Raynaud1974}, its restriction to inertia has the form
\[
 \psi|_{I_{\mathfrak p}}=\omega_2^{r_0+7r_1},
 \qquad r_0,r_1\in\{0,1\}.
\]
The character has level dividing two because it is a rank-one character
of the full local Galois group.  Here $\omega_2$ is a fundamental
character of level two, normalized so
that $\omega_2^{8}=\bar\chi_7|_{I_{\mathfrak p}}$.
The two mixed possibilities must be excluded globally.  Put
$\varepsilon=(1+\sqrt5)/2$ and
\[
 u=\varepsilon^8=13+21\varepsilon.
\]
The unit $u$ is totally positive and satisfies
\[
 u\equiv1\pmod{3(\sqrt5)},\qquad u\equiv-1\pmod{7}.
\]
Global reciprocity applied to $u$ shows that the value of $\psi$ on its
local reciprocity image at $\mathfrak p$ is $1$: at every other finite
place this follows from the conductor bounds just proved, and at both
real places it follows from total positivity.  The displayed inertial
character gives that value as $(-1)^{r_0+r_1}$.  Hence $r_0=r_1$.
If both are zero then $\psi$ is unramified at $\mathfrak p$; if both are
one then $\psi^{-1}\bar\chi_7$ is unramified there.  Interchanging the two
characters if necessary, we may therefore choose $\psi$ such that
\begin{equation}\label{eq:character-conductor}
 \operatorname{cond}(\psi)\mid3(\sqrt5).
\end{equation}
Including both real places gives a modulus containing the conductor of
$\psi$ regardless of its signs.  Its narrow ray group is
$C_4\times C_2$, so there are eight possible characters, all of which are
included in the comparison.

At $q=29$, take
$\mathfrak q=(29,\sqrt5-11)$; the conjugate prime gives the conjugate
packet and the same common-root test.  The exceptional septic polynomial
is squarefree modulo $29$ with factor degrees $[1,3,3]$.  Among
$\eta\in\F_{29}\setminus\{0,1\}$, the polynomial $\Phi(T)-\eta$ has
this pattern precisely for $\eta=10,14,24,28$.

The branch residues do not give this pattern.  If $29\mid X$, scale the
five-root cluster at $T=0$ using $\ord_{29}(\eta)=5\ord_{29}(X)$.
Its separable reduction has degrees $[1,2,2]$, since $29\equiv-1\pmod5$;
the complementary quadratic has degrees $[1,1]$.  If $29\mid Y$, the
corresponding three-root cluster at $T=1$ has degrees $[1,2]$, and the
complementary quartic has degrees $[1,1,2]$.  Both cases therefore have
pattern $[1,1,1,2,2]$.  If $29\mid Z$, the scaled coordinate $U=ZT$
gives reduction $15U^7-X^5$; its pattern is $[1,1,1,1,1,1,1]$ or $[7]$
because $\F_{29}$ contains the seventh roots of unity.  Hensel's lemma
applies to each scaled separable reduction.  Thus only the four ordinary
values remain.  By \eqref{eq:PVT-map} they correspond to
$t_0=14,10,25,15$.  Direct point counting on
\[
                y^2=5x^6-12x^5+10t_0x^3+t_0^2
\]
gives the following polynomials over $\F_7$ satisfied by the two
conjugate real-multiplication Frobenius traces.  Explicitly, if
$N_1=\#C_{t_0}(\F_{29})$, $N_2=\#C_{t_0}(\F_{29^2})$,
$A=30-N_1$, and $B=(A^2-(842-N_2))/2$, the degree-four Frobenius
polynomial is $X^4-AX^3+BX^2-29AX+29^2$, and the two traces satisfy
$T^2-AT+(B-58)$.  The table records the latter polynomial modulo $7$:
\begin{equation}\label{eq:curve-polynomials}
\begin{array}{c|c|c}
\eta&t_0&\text{curve polynomial}\\ \hline
10&14&T^2+6T+6\\
14&10&T^2+6T+4\\
24&25&T^2+4\\
28&15&T^2+2T+3.
\end{array}
\end{equation}
For clarity, the next list is an \emph{annihilating family}; its members are
not asserted to be irreducible or squarefree.  For each of the four ambient
cuspidal spaces, reduce the characteristic polynomial of the Hecke operator
$T_{\mathfrak q}$ modulo $7$ on every packet surviving the conservative
$T_2$ condition.  The preceding dimension accounting proves that every
residual eigenvalue $a_{\mathfrak q}(f)\bmod\mathfrak p$, for every relevant
newform and each coefficient prime $\mathfrak p\mid7$ realizing the required
Frey congruence, is a root of
one of the first four or the sixth displayed factors.  We retain the factors as they
occur in the packet calculation, so they need not be squarefree or minimal.
The reducible characters contribute the possible traces
$\psi(\mathfrak q)+\psi(\mathfrak q)^{-1}$.  The resulting complete
annihilating family is
\begin{equation}\label{eq:test-polynomials}
 T,\quad T^2,\quad T^2+5T+2,\quad T^2+3T+4,\quad T-2,\quad T-5.
\end{equation}
Here $T^2$ and $T^2+3T+4=(T-2)^2$ deliberately record multiplicity; for
the resultant test they have the same roots as $T$ and $T-2$.  For the specified prime $\mathfrak q$, the ray class has order two:
a totally positive generator is $6-\varepsilon$, whose square is
congruent to the totally positive unit $\varepsilon^2$ modulo
$3(\sqrt5)$; its class is nontrivial, since modulo $3$ the generator is
$-\varepsilon=\varepsilon^5$, whereas every totally positive unit is an
even power of $\varepsilon$, whose order modulo $3$ is $8$.
Thus the reducible traces are
$2$ and $-2$, covered by the displayed linear factors.  The
$24$ pairwise resultants of \eqref{eq:curve-polynomials} and
\eqref{eq:test-polynomials} form the matrix
\begin{equation}\label{eq:resultants}
 \begin{pmatrix}
 6&1&5&1&1&5\\
 4&2&3&1&6&3\\
 4&2&6&1&1&1\\
 3&2&6&2&4&3
 \end{pmatrix}\pmod7.
\end{equation}
Every entry is nonzero.  Thus the exceptional field is incompatible with
both the irreducible and reducible residual alternatives.

\begin{proposition}\label{prop:septic-sector}
The sector $[7]$ is empty.
\end{proposition}

\section{Proof of the main theorem}

\begin{proof}[Proof of Theorem~\ref{thm:main}]
Suppose that a nonzero primitive solution exists and reorder it as in
\eqref{eq:DS-order}.  The descent algebra $A_\eta$ belongs to exactly one
of the four sectors in \eqref{eq:four-sectors}.  Proposition
\ref{prop:rational-sector} excludes $[1,6]$.  Proposition
\ref{prop:quadratic-sector} excludes $[2,5]$.  Proposition
\ref{prop:quartic-sector} excludes $[3,4]$.  Proposition
\ref{prop:septic-sector} excludes $[7]$.  Every possible factorization is
therefore impossible, a contradiction.
\end{proof}

\section{Computational certification}

The computer-assisted portions of the proof are finite statements over
$\Q$, number fields, finite fields, and truncated local rings.  They are
organized so that discovery is not used as evidence: every accepted output
is checked against exact input data and against the mathematical predicate
stated in the relevant proposition.  The certificate archive serves two
different purposes.  It independently reconstructs finite calculations from
the cited arguments of Dahmen--Siksek, and it supplies the load-bearing arithmetic
for the new seven-field elimination.  The provenance column below keeps
those purposes distinct.

\subsection{Exact arithmetic layer}

The exact layer checks polynomial identities, resultants, discriminants,
field embeddings, ideal-class representatives, divisor relations, point
incidence, and rational substitutions.  Integer and rational calculations
are repeated without floating-point arithmetic.  Manifests identify the
input bytes and the corresponding outputs.

\subsection{Finite exhaustion layer}

Whenever a finite search is used as an exhaustion, the searched set is
defined before enumeration.  This applies to local factor patterns, Selmer
classes, reduced curve points, Mordell--Weil coefficient classes, residue
trees, ray characters, newform packets, and resultant comparisons.  A
residue-tree node denotes a complete congruence class.  A rejected node is
accompanied by a uniform congruence or valuation obstruction; it is not a
sample point.

\subsection{Local analytic layer}

Formal-group and Coleman calculations record their working precision,
unit pivots, ranks, determinants, and higher-precision replays.  Divisions
are made only by certified units or with their exact loss of precision
recorded.  Simple-root conclusions are accompanied by a nonzero reduced
Jacobian determinant.  Singular classes are subdivided or normalized until
the stated conclusion follows.
For the depth-two residue trees of Section~\ref{subsec:twelve-local}, the
effective precision certificate is the independent exact checker
\path{audit/work/v3_audit/section7_independent_local.py} (prior-evidence
job \texttt{section7\_precision\_bounds}), which recomputes all ten trees
with explicit visibility checks at every rejected node; the generic
precision bound that the original helper
\path{verify_section7_formal_group_lemma_7_5_v2.py} assigns to
conjugate-branch coordinates is not relied upon, and that helper is
unchanged in the present companion.

\subsection{Software and trust boundary}

The calculations use PARI/GP \cite{PARI} for exact number-field arithmetic,
SageMath \cite{Sage} for curve and finite-field calculations, and Magma
\cite{Magma} for descent,
Mordell--Weil, saturation, newform, and Chabauty routines.  Their role is
the same as that of large integer arithmetic or a table of modular forms in
a conventional proof: the inputs, outputs, and verification conditions are
part of the mathematical record.  The argument also imports the general
theorems cited in the bibliography, rather than attempting to reprove them
from first principles.  In the cubic--quartic sector, the rational and
twisted Mordell--Weil and torsion results in
\cite[Lemma~7.3]{DahmenSiksek} remain explicit inputs.
The finite-field $173$ Kummer argument supplies the separate
$2$-saturation step; it does not replace those rank and torsion inputs.
Seven archived Magma inputs, including both full-group computations,
were executed through the official Magma calculator, version V2.29-10.
Their exact inputs, outputs and mathematical checks are included in
the companion verification materials.  In both full-group computations,
the returned \texttt{finite\_index} and \texttt{proved} flags are true,
and the returned rank bounds are respectively $1$ and $2$; the scripts
also verify the named generators in the returned groups.  These executions use Magma's
algorithms; they are not independent implementations.  The numerical
height comparison supplements, and does not replace, the separate exact
height bounds.

\paragraph{Substantial AI-assisted research and writing.}
OpenAI Codex was used extensively in developing and checking mathematical
arguments, computational investigation, program and certificate development,
verification design, debugging, release preparation, and drafting and
revising the manuscript.  OpenAI ChatGPT also supported the research and
writing.  Anthropic Claude acted as an independent auditor throughout:
it reproduced the finite data from published inputs, reviewed every
draft, and checked the numbering and content of the cited results
against their sources.  Its findings led to the model-intrinsic formal-group
argument at the prime above $7$ and to the scope statement for
solutions with a unit coordinate in Section~1, and it reconstructed
Dahmen--Siksek's database-free identification of $\Q(\sqrt{-35})$ as a
certified calculation (an earlier draft wrongly credited that
identification as new; the error was found in pre-submission review); it supplied the
characterization of $\Phi$ and the derivation of the ramification
statement in Section~\ref{sec:descent}; it independently verified the
bitangent code data used in Appendix~\ref{app:rank-proof} (the $315$
syzygetic tetrads, the rank $21$, and the weight enumerator); and it
verified the digest, byte count and integrity check of the companion
archive and prepared the final arXiv source.
Peter Chocian
originated the project's discovery approach and research direction,
supplied key inputs, reviewed the evidence, and accepts responsibility for
all claims.  AI-generated suggestions and reviews are not mathematical
certificates: the proof rests on the arguments given here, the cited
results and the specified computations.  These AI-assisted reviews do not
constitute independent human peer review or proof-assistant formalization.

\paragraph{Licence.}
Copyright \copyright\ 2026 Peter Chocian.  The author's original manuscript
and documentation are licensed under
\href{https://creativecommons.org/licenses/by/4.0/}{Creative Commons
Attribution 4.0 International}; the author's original program code is
licensed under the MIT License.  These grants exclude third-party papers,
software and other material, whose rights remain with their respective
holders.  In particular, they do not grant permission to redistribute
the third-party material in the frozen companion archives.

\subsection{Companion certificate archive}

The programs, exact inputs, certificates and verification records
are available from the repository
\href{https://github.com/bbpcho/primitive-357}{\texttt{bbpcho/primitive-357}}.
The computational companion is
\path{PRIMITIVE_357_REPLAY_COMPANION_2026-09-23_V4.zip}
(298483791 bytes). Its public release tag is
\begin{center}
\href{https://github.com/bbpcho/primitive-357/releases/tag/replay-companion-2026-09-23.1}{\texttt{replay-companion-2026-09-23.1}}.
\end{center}
Its SHA-256 digest is
\begin{quote}\small\ttfamily
ecc39d2944092fc2efbfdd17ffe9861e6c1ca68247b06af80bd7bd4bda166c37
\end{quote}
with distribution manifest \texttt{c376222d\ldots} and logical input
manifest
\begin{quote}\small\ttfamily
91a2bb027747b6663b5e075d31001e8ee6d12eec3025dd53f86818afffd2ef5d
\end{quote}
($2056$ logical files, $74$ separately acquired public inputs,
dependency lock \texttt{478111f4\ldots}).
It supersedes the archive with tag \texttt{replay-companion-2026-09-22.1},
which withheld the author's internal reports as author-held inputs and
therefore could not be replayed from public material, the replayed
distribution with tag \texttt{replay-companion-2026-09-15.2}, and the
two withdrawn archives with historical tags
\begin{center}\small
\texttt{verified-2026-09-14.1}\\
\texttt{first-submission-2026-09-15.1}.
\end{center}
Relative to the replayed distribution, the present archive removes
$72$ exploratory working reports (narrative documents that are inputs
of no executed calculation) and regenerates the manifests, indices,
ledgers and certificates that recorded their digests; ten verifiers
that certified their own narrative report as an artifact no longer do
so, and every pinned manifest row count is adjusted by the
corresponding removal.  No mathematical acceptance predicate, exact
arithmetic input, or certified mathematical value is changed.  The file \path{RELEASE_DELTA_2026-09-23.json} lists
every removed document, every edited verifier and record, and every
regenerated file.  The current paper and its small arXiv source
archive are separate assets; the earlier manuscript draft inside the
computational records is retained as provenance.

Because its expanded tree differs from the recorded one, the complete
declared suite was rerun from the public archive and its $74$ public
inputs on an independent machine (Ubuntu~24.04, aarch64; conda-forge
SageMath~10.9 with Python~3.12; PARI/GP~2.15.4 from the Ubuntu
package), with network access blocked and no access to the author's
workspaces.  The result is
\texttt{PASS\_FRESH\_EXTERNALIZED\_COMPLETE\_REPLAY}: all $13$ sector and
interface records, all $65$ prior jobs, and all $16$ rank-local checks
passed, in $2562$ seconds on 22 September 2026, with integrated result
digest \texttt{e1941808\ldots} and independent envelope
\texttt{6ad3f186\ldots}.  The record archive
\path{PRIMITIVE_357_V4_REPLAY_RECORD_SPARK_2026-09-22.zip}
(1082797043 bytes, SHA-256 \texttt{dfa88210\ldots}) is a release asset
alongside the companion.  Two runtime requirements were observed and
are documented in the companion: the exact-comparison job
\texttt{reconstruct\_raw66} requires PARI~2.15.4 (version~2.17
returns equivalent but not byte-identical representatives, and the
gate fails by design), and the replay must be launched with standard
input detached from a terminal.

The companion stores repeated evidence once and reconstructs the exact
logical input layout for replay. Its explicit dependency lock gives the
source, size and SHA-256 digest of every separately acquired input.
The Dahmen--Siksek, PVT and BPS papers and unlicensed upstream code
snapshots are acquired from their sources rather than redistributed.
Project modifications to pinned upstream source files are reconstructed
locally and checked against their exact expected digests. Third-party
material retained under its own licence is identified in the companion's
attribution notice. Seven historical executables used only for byte
identity checks in the earlier package are omitted; their identities and
the affected provenance checks are recorded explicitly. No new execution
of those binaries or of the historical Magma jobs is claimed.

The public integrity check, run from the extracted companion root, is
\begin{center}
\texttt{python3 -B scripts/verify\_companion.py}.
\end{center}
It verifies the exact public file set and logical input mapping.
The separate script \path{scripts/replay_companion.py}, supplied with
explicit software and input paths, runs the full declared suite from
new output directories. The recorded replay of 15 September was run from a clean
extraction of the replayed distribution under the same isolation, on
macOS with the recorded SageMath and PARI, with the same outcome; the
rank-local checks consume five freshly generated inputs.  The expanded
logical inputs and separately acquired inputs are authenticated before
and after execution.
The inherited internal identifier \texttt{verified-2026-09-14.1} is
retained only for verifier compatibility; the input manifest was
recomputed for the present distribution.

The companion guide gives acquisition, reconstruction and replay
commands, runtime requirements, dependency indexes, successful logs and
retained imported premises. The rank route includes the local-image
and fake-kernel stopping proofs in Appendix~\ref{app:rank-proof},
the same-literal-$c$ comparison and the $H_0/H_1$ index statement.
Supplementary argument-audit records and the seven earlier Magma
executions are identified separately. Integrity checking alone is not
an arithmetic replay.
The principal correspondence is:
\begin{center}
\begin{tabular}{@{}>{\raggedright\arraybackslash}p{.22\textwidth}
                    >{\raggedright\arraybackslash}p{.20\textwidth}
                    >{\raggedright\arraybackslash}p{.48\textwidth}@{}}
\toprule
mathematical assertion & provenance & certificate family\\
\midrule
four exhaustive factor sectors
  & Dahmen--Siksek; reconstructed
  & signed global-algebra superselection\\
Proposition~\ref{prop:rational-sector}
  & Dahmen--Siksek; reconstructed
  & rational-factor composition and its Mordell--Weil, sieve, and local packages\\
Proposition~\ref{prop:quadratic-sector}
  & Dahmen--Siksek; reconstructed
  & database-free quadratic router (Dahmen--Siksek, Proposition~4.2) and quadratic-factor reconstruction\\
Proposition~\ref{prop:quartic-sector}
  & Dahmen--Siksek; new formal-group certification
  & quartic router, Mordell--Weil basis, sieve, and intrinsic formal-group packages\\
Theorem~\ref{thm:rational-t} and Appendix~\ref{app:Cplus}
  & new
  & Fano-cover, canonical-quartic, rational-parameter, and split-prime packages\\
exceptional-field elimination
  & new, using cited modular theorems
  & Putz--PVT interface and finite $q=29$ comparison packages\\
Proposition~\ref{prop:septic-sector}
  & new terminal composition
  & irreducible-sector composition package\\
Theorem~\ref{thm:main}
  & new global composition
  & complete four-sector proof graph and terminal certificate\\
\bottomrule
\end{tabular}
\end{center}
The archive distinguishes identity checks, arithmetic replay, and the
mathematical implications of the resulting certificates.  A checksum
establishes identity, not a theorem.  The adopted Python, PARI/GP and
SageMath programs reconstruct the specified arithmetic layers, with errors
and precision failures treated as failures.  The companion guide
distinguishes imported literature results, authenticated Magma execution
records, and independently reconstructed calculations.
The global rank and local stopping conclusions use the stated general
theorems and verified hypotheses as well as the replayed arithmetic.

\section{Concluding perspective}

The proof is not driven by the size of a putative solution.  It represents
that solution simultaneously by a descent algebra, local factor labels,
divisor classes, finite quotient coordinates, and $p$-adic analytic data.
Each representation imposes an exact compatibility condition.  The global
contradiction comes from the empty intersection of those conditions.

Everything arithmetic in the proof is specific to the signature: the cover
$\Phi$, the seven-field list, the Fano resolvent and $C_+$, the Frey family,
and every certificate.  What is useful beyond the present signature is the
organization.  It separates three tasks that are often entangled: enumerating the possible algebraic
states, transporting a state to a curve where a finitely generated group
acts, and testing the remaining local states with analytic functions.  It
also makes the proof modular: a computational component can be replaced by an
independent implementation without changing the surrounding argument, while
the finite interfaces remain explicit.

\appendix

\section{The plane quartic in the pure septic case}\label{app:Cplus}

Let $K_7=\Q(s)$ with $s^2=-7$.  In the monomial order
\[
 x^4,x^3y,x^3z,x^2y^2,x^2yz,x^2z^2,xy^3,xy^2z,xyz^2,xz^3,
 y^4,y^3z,y^2z^2,yz^3,z^4,
\]
write the coefficient of each monomial as $(U+sW)/2$.  The pairs $(U,W)$
for $C_+$ are recorded below with the monomial repeated on every row, so
that each coefficient can be quoted and checked independently of the
surrounding order.
\begin{landscape}
\small
\renewcommand{\arraystretch}{1.22}
\begin{longtable}{@{}c>{\raggedleft\arraybackslash}p{82mm}>{\raggedleft\arraybackslash}p{82mm}@{}}
\caption{Exact coefficients of the plane quartic $C_+$.  The coefficient
of $m$ is $(U+sW)/2$.}\label{tab:cplus-quartic-coefficients}\\
\toprule
monomial $m$ & $U$ & $W$\\
\midrule
\endfirsthead
\multicolumn{3}{c}{\tablename\ \thetable\ (continued)}\\
\toprule
monomial $m$ & $U$ & $W$\\
\midrule
\endhead
$x^4$       & $526589143431946875/69798278$   & $-114868376351859375/69798278$\\
$x^3y$      & $-1271586888400500/3172649$     & $-870445613016000/3172649$\\
$x^3z$      & $72385949750625/34899139$        & $-402884946883125/34899139$\\
$x^2y^2$    & $-666163656476955/34899139$      & $264021753009735/34899139$\\
$x^2yz$     & $-74662865230275/34899139$       & $45051674372775/34899139$\\
$x^2z^2$    & $-20453292214425/69798278$       & $1113321269025/69798278$\\
$xy^3$      & $77961057423858/174495695$       & $1554848756622/34899139$\\
$xy^2z$     & $5974117172919/34899139$         & $534487214709/34899139$\\
$xyz^2$     & $251941629933/34899139$          & $209684218779/34899139$\\
$xz^3$      & $-13800316635/69798278$          & $9656077725/69798278$\\
$y^4$       & $-11467163171457/8724784750$     & $-8293918213539/8724784750$\\
$y^3z$      & $-192270890133/174495695$        & $-111061125567/174495695$\\
$y^2z^2$    & $-4669897617/31726490$           & $-23997395583/158632450$\\
$yz^3$      & $4101243048/174495695$           & $-629073270/34899139$\\
$z^4$       & $2$                              & $0$\\
\bottomrule
\end{longtable}
\end{landscape}
The equation and its three first partial derivatives have no common
projective zero over $K_7$; this is checked by exact elimination.  As an
independent good-reduction certificate, both reductions at the primes above
$23$ are smooth plane quartics.  Thus $C_+$ is smooth, nonhyperelliptic, and
has genus $3$.  The two reductions and their common zeta numerator are
recorded in Appendix~\ref{app:p23-data}.
The five points in Theorem~\ref{thm:rational-t} are
\begin{align*}
P_1={}&[-113-139s:32925+11475s:-138600-57960s],\\
P_2={}&[1109-305s:4325-19625s:163800-37800s],\\
P_3={}&[-1+s:50:0],\\
P_4={}&[5+s:400:0],\\
P_5={}&[-157167-154941s:-82225175+2993875s:186832800].
\end{align*}
The homogeneous forms $G_t$ and $H_t$ have degree $5$.  For each monomial
$m=x^ay^bz^c$ of degree $5$, write
\[
 [m]G_t=\frac{U_{G,m}+sW_{G,m}}2,
 \qquad
 [m]H_t=\frac{U_{H,m}+sW_{H,m}}2.
\]
Table~\ref{tab:t-parameter-coefficients} gives the four rational components
for every monomial; a component is displayed as separate numerator and
positive denominator.  These data define the parameter completely, with no
external coefficient file required.

\newcommand{\ExactLedgerBare}{%
  \begin{center}%
  \begin{tabular}{@{}cc>{\raggedright\arraybackslash}p{170mm}@{}}}
\newcommand{\ExactLedgerHead}{%
  \ExactLedgerBare
  \toprule
  monomial & component & rational value\\
  \midrule}
\newcommand{\ExactLedgerTail}{%
  \bottomrule
  \end{tabular}%
  \end{center}}
\newcommand{\ExactLedgerNext}{%
  \ExactLedgerTail
  \end{landscape}
  \begin{landscape}
  \small
  \renewcommand{\arraystretch}{2.5}
  \begin{center}\tablename\ \ref{tab:t-parameter-coefficients} (continued)\end{center}
  \ExactLedgerHead}
\begin{landscape}
\small
\renewcommand{\arraystretch}{2.5}
\begin{center}
\captionof{table}{Exact coefficient ledger for the inverse parameter $t=-G_t/H_t$.
For each monomial and each form, the two successive rows are the rational
components $U$ and $W$ of the coefficient $(U+sW)/2$.}
\label{tab:t-parameter-coefficients}
\end{center}
\ExactLedgerHead
\multicolumn{3}{c}{$G_t$}\\
$x^{5}$ & $U$ & $\frac{\mathtt{8724124647223604676837903272778053868067651476251979316420946130788203125}}{\mathtt{1262470913738019638514233667224982865889496755254608725372117056}}$\\
$x^{5}$ & $W$ & $\frac{\mathtt{104539127651167761178849503246013584932810386714540106341619861276432421875}}{\mathtt{34086714670926530239884309015074537379016412391874435585047160512}}$\\
$x^{4}y$ & $U$ & $\frac{\mathtt{160516210802709714993563463849330968569737125845236059453160939760278125}}{\mathtt{516465373801917124846731954773856626954794127149612660379502432}}$\\
$x^{4}y$ & $W$ & $\frac{\mathtt{-1516029210727323321897909052657899390133429038437597834516845689254996875}}{\mathtt{4648188364217254123620587592964709642593147144346513943415521888}}$\\
$x^{4}z$ & $U$ & $\frac{\mathtt{-3431261016436648118361864834749807161605723145608327043665883055988296875}}{\mathtt{536865756067092851278177866987423963719508495172022360464492778064}}$\\
$x^{4}z$ & $W$ & $\frac{\mathtt{-14405986606016247091870808274044492866667442818955446375070000269331640625}}{\mathtt{536865756067092851278177866987423963719508495172022360464492778064}}$\\
$x^{3}y^{2}$ & $U$ & $\frac{\mathtt{-156765051716182227120181782804196796183006994771851219874584685103262375}}{\mathtt{6391259000798724419978307940326475758565577323476456672196342596}}$\\
$x^{3}y^{2}$ & $W$ & $\frac{\mathtt{-2797425838906852452662912431708042341593603601309085420184135364665875}}{\mathtt{3195629500399362209989153970163237879282788661738228336098171298}}$\\
\ExactLedgerNext
$x^{3}yz$ & $U$ & $\frac{\mathtt{-2971418364998561881697119480380162120182637893006270782575715726325204375}}{\mathtt{805298634100639276917266800481135945579262742758033540696739167096}}$\\
$x^{3}yz$ & $W$ & $\frac{\mathtt{581935828601695609265701070325316097436069202977166807199027638486990625}}{\mathtt{805298634100639276917266800481135945579262742758033540696739167096}}$\\
$x^{3}z^{2}$ & $U$ & $\frac{\mathtt{-1101411663992169427898989261604416930376364487448681819065769508346053125}}{\mathtt{11274180877408949876841735206735903238109678398612469569754348339344}}$\\
$x^{3}z^{2}$ & $W$ & $\frac{\mathtt{-107499314624920129138097362917410399320732876511214060737904168396284375}}{\mathtt{3758060292469649958947245068911967746036559466204156523251449446448}}$\\
$x^{2}y^{3}$ & $U$ & $\frac{\mathtt{456938855157174881485528038353970608056813975973606988967508412987185}}{\mathtt{4260839333865816279985538626884317172377051548984304448130895064}}$\\
$x^{2}y^{3}$ & $W$ & $\frac{\mathtt{1661711353625849642579873501624741967233922652605715624836234484805605}}{\mathtt{12782518001597448839956615880652951517131154646952913344392685192}}$\\
$x^{2}y^{2}z$ & $U$ & $\frac{\mathtt{8474129667022123046619324483315222687658910122135070504455178101736425}}{\mathtt{115042662014377039559609542925876563654180391822576220099534166728}}$\\
$x^{2}y^{2}z$ & $W$ & $\frac{\mathtt{10719261028668655910991493734095990759234174142059301575470250588670875}}{\mathtt{268432878033546425639088933493711981859754247586011180232246389032}}$\\
\ExactLedgerNext
$x^{2}yz^{2}$ & $U$ & $\frac{\mathtt{4788522504016908727437175426249647578416822156275084241334728620014875}}{\mathtt{16911271316113424815262602810103854857164517597918704354631522509016}}$\\
$x^{2}yz^{2}$ & $W$ & $\frac{\mathtt{5591416257739747140557321332572949584770401809552651792943380516721625}}{\mathtt{1879030146234824979473622534455983873018279733102078261625724723224}}$\\
$x^{2}z^{3}$ & $U$ & $\frac{\mathtt{-673613683718293366731089090657452594859133231945568100511338992953086875}}{\mathtt{4261640371660583053446175908146171424005458434675513497367143672272032}}$\\
$x^{2}z^{3}$ & $W$ & $\frac{\mathtt{727154975549102369129278091212463518913101596132005437190890995061885625}}{\mathtt{4261640371660583053446175908146171424005458434675513497367143672272032}}$\\
$xy^{4}$ & $U$ & $\frac{\mathtt{181934888585282020021801499997708888564901966896834717514388222336671}}{\mathtt{127825180015974488399566158806529515171311546469529133443926851920}}$\\
$xy^{4}$ & $W$ & $\frac{\mathtt{-66576524945702734300367034318330080309880903641645008587735355126139}}{\mathtt{127825180015974488399566158806529515171311546469529133443926851920}}$\\
$xy^{3}z$ & $U$ & $\frac{\mathtt{55752557283144284245053396819690398445556264921670780430266851605377}}{\mathtt{134216439016773212819544466746855990929877123793005590116123194516}}$\\
$xy^{3}z$ & $W$ & $\frac{\mathtt{-154753525348486993811829134912142488930864130682803774941942077141595}}{\mathtt{402649317050319638458633400240567972789631371379016770348369583548}}$\\
\ExactLedgerNext
$xy^{2}z^{2}$ & $U$ & $\frac{\mathtt{45466049893188675406917805349188754066362340943964804016941934609315}}{\mathtt{1537388301464856801387509346373077714287683417992609486784683864456}}$\\
$xy^{2}z^{2}$ & $W$ & $\frac{\mathtt{-90644573939421031699724276084922394905211688255008129426169575178615}}{\mathtt{1537388301464856801387509346373077714287683417992609486784683864456}}$\\
$xyz^{3}$ & $U$ & $\frac{\mathtt{86176315804162437100883836603358200621380221474103257474111818725307525}}{\mathtt{6392460557490874580169263862219257136008187652013270246050715508408048}}$\\
$xyz^{3}$ & $W$ & $\frac{\mathtt{-14901625678223850735415766370606467751263679923368054621027598776871975}}{\mathtt{6392460557490874580169263862219257136008187652013270246050715508408048}}$\\
$xz^{4}$ & $U$ & $\mathtt{0}$\\
$xz^{4}$ & $W$ & $\mathtt{0}$\\
$y^{5}$ & $U$ & $\frac{\mathtt{-50208669906374986341132909015146131126007030919811841716909}}{\mathtt{18313514842869689191983526987088351258653049645369350436400}}$\\
$y^{5}$ & $W$ & $\frac{\mathtt{-118965057882822154213596381597899423461490949136300188722323}}{\mathtt{91567574214348445959917634935441756293265248226846752182000}}$\\
\ExactLedgerNext
$y^{4}z$ & $U$ & $\frac{\mathtt{-317869767610131557388806079392661467805959663406188675965981}}{\mathtt{164821633585827202727851742883795161327877446808324153927600}}$\\
$y^{4}z$ & $W$ & $\frac{\mathtt{-20431957573101268107747459729242315992574863384255199500651}}{\mathtt{32964326717165440545570348576759032265575489361664830785520}}$\\
$y^{3}z^{2}$ & $U$ & $\frac{\mathtt{-90218411306089523014211143406416603450262335284699082090763}}{\mathtt{346125430530237125728488660055969838788542638297480723247960}}$\\
$y^{3}z^{2}$ & $W$ & $\frac{\mathtt{-6384772760573948911903427720751451666967418352649934952949}}{\mathtt{69225086106047425145697732011193967757708527659496144649592}}$\\
$y^{2}z^{3}$ & $U$ & $\frac{\mathtt{34563607476891939880909708193414566015508357013210019043845}}{\mathtt{13083541274042963352536871350115659906206911727644771338772888}}$\\
$y^{2}z^{3}$ & $W$ & $\frac{\mathtt{-150431184828356847118153754501573195221720737251858300004385}}{\mathtt{13083541274042963352536871350115659906206911727644771338772888}}$\\
$yz^{4}$ & $U$ & $\mathtt{0}$\\
$yz^{4}$ & $W$ & $\mathtt{0}$\\
$z^{5}$ & $U$ & $\mathtt{0}$\\
$z^{5}$ & $W$ & $\mathtt{0}$\\
\ExactLedgerTail
\end{landscape}
\begin{landscape}
\small
\renewcommand{\arraystretch}{2.5}
\ExactLedgerBare
\multicolumn{3}{c}{\tablename\ \ref{tab:t-parameter-coefficients} (continued)}\\
\toprule
monomial & component & rational value\\
\midrule
\multicolumn{3}{c}{$H_t$}\\
$x^{5}$ & $U$ & $\frac{\mathtt{417540987588276327046139642388996781228921135238823876465270116076953125}}{\mathtt{55097060351713141551537137416607549642806701064452212691347781}}$\\
$x^{5}$ & $W$ & $\frac{\mathtt{271400501016398695076663568050301774206517949977640898218621509239453125}}{\mathtt{55097060351713141551537137416607549642806701064452212691347781}}$\\
$x^{4}y$ & $U$ & $\frac{\mathtt{40800518279237110617651643358322019052974596058651349157620400597734375}}{\mathtt{40070589346700466582936099939350945194768509865056154684616568}}$\\
$x^{4}y$ & $W$ & $\frac{\mathtt{-10475518577744892148533246091047638248959642317009125512752949147203125}}{\mathtt{40070589346700466582936099939350945194768509865056154684616568}}$\\
$x^{4}z$ & $U$ & $\frac{\mathtt{295466106564310819299307055245131237369729837129047525997756150038046875}}{\mathtt{13884459208631711670987358628985102509987288668241957598219640812}}$\\
$x^{4}z$ & $W$ & $\frac{\mathtt{-97401726931434571033238656526295780061095710676931455291585333023515625}}{\mathtt{13884459208631711670987358628985102509987288668241957598219640812}}$\\
$x^{3}y^{2}$ & $U$ & $\frac{\mathtt{-2408541684482196326060738206798781548127364231879824910739247944739375}}{\mathtt{73462747135617522068716183222143399523742268085936283588463708}}$\\
$x^{3}y^{2}$ & $W$ & $\frac{\mathtt{-1401617439695752490670177104164771958374512006836961227347342807905625}}{\mathtt{220388241406852566206148549666430198571226804257808850765391124}}$\\
\ExactLedgerNext
$x^{3}yz$ & $U$ & $\frac{\mathtt{-88727646196045391066194147279337983419623095543345635287685584174321875}}{\mathtt{13884459208631711670987358628985102509987288668241957598219640812}}$\\
$x^{3}yz$ & $W$ & $\frac{\mathtt{-294670882430653452222330070241389024327059653868933829299253113096875}}{\mathtt{661164724220557698618445648999290595713680412773426552296173372}}$\\
$x^{3}z^{2}$ & $U$ & $\frac{\mathtt{-143378643984727743099590687262597995420899449430994725195546663512734375}}{\mathtt{1166294573525063780362938124834748610838932248132324438250449828208}}$\\
$x^{3}z^{2}$ & $W$ & $\frac{\mathtt{-18365708170095511831980056438251994847190327220498321285191544780928125}}{\mathtt{166613510503580540051848303547821230119847464018903491178635689744}}$\\
$x^{2}y^{3}$ & $U$ & $\frac{\mathtt{8623734371967307823196434999403662938951554222083721534345295726825}}{\mathtt{110194120703426283103074274833215099285613402128904425382695562}}$\\
$x^{2}y^{3}$ & $W$ & $\frac{\mathtt{18881132466481791978253130146174121549567643122623579546749820507925}}{\mathtt{110194120703426283103074274833215099285613402128904425382695562}}$\\
$x^{2}y^{2}z$ & $U$ & $\frac{\mathtt{28562333150244501488893362369966137889574552762163770883392632257375}}{\mathtt{661164724220557698618445648999290595713680412773426552296173372}}$\\
$x^{2}y^{2}z$ & $W$ & $\frac{\mathtt{127302955155020767473662087940112338634670901473338060096039562505625}}{\mathtt{1983494172661673095855336946997871787141041238320279656888520116}}$\\
\ExactLedgerNext
$x^{2}yz^{2}$ & $U$ & $\frac{\mathtt{111047168050090146883232536414177868796628159548247958136634257244625}}{\mathtt{583147286762531890181469062417374305419466124066162219125224914104}}$\\
$x^{2}yz^{2}$ & $W$ & $\frac{\mathtt{808931698949954129585743924035840186651521792034283998632202423247375}}{\mathtt{194382428920843963393823020805791435139822041355387406375074971368}}$\\
$x^{2}z^{3}$ & $U$ & $\frac{\mathtt{56671479059928976818861058588328919604601251570358478252116683982508125}}{\mathtt{440859348792474108977190611187534974897116389794018637658670035062624}}$\\
$x^{2}z^{3}$ & $W$ & $\frac{\mathtt{71221107904627237798651209022651158397866104168849046635811093612975625}}{\mathtt{440859348792474108977190611187534974897116389794018637658670035062624}}$\\
$xy^{4}$ & $U$ & $\frac{\mathtt{15345599112192921388150576223890339236281253816603221545634911025}}{\mathtt{18365686783904380517179045805535849880935567021484070897115927}}$\\
$xy^{4}$ & $W$ & $\frac{\mathtt{-32203376966478498296624374226292138386494075379549015420791430453}}{\mathtt{55097060351713141551537137416607549642806701064452212691347781}}$\\
$xy^{3}z$ & $U$ & $\frac{\mathtt{3144967508609794715796497440988774261265066246157112647816548077300}}{\mathtt{3471114802157927917746839657246275627496822167060489399554910203}}$\\
$xy^{3}z$ & $W$ & $\frac{\mathtt{-1516919058667068918402436722904874293272390262093914022875222741380}}{\mathtt{3471114802157927917746839657246275627496822167060489399554910203}}$\\
\ExactLedgerNext
$xy^{2}z^{2}$ & $U$ & $\frac{\mathtt{1438698269678151371696913059361694896892235546906064110725007289965}}{\mathtt{6626673713210589661153057527470162561584842318933661580968464933}}$\\
$xy^{2}z^{2}$ & $W$ & $\frac{\mathtt{-80096480704524798441566620354536123544154597802973268379738357665}}{\mathtt{946667673315798523021865361067166080226406045561951654424066419}}$\\
$xyz^{3}$ & $U$ & $\frac{\mathtt{3102145252841255952780604187946926413931338657791125868394786111962775}}{\mathtt{220429674396237054488595305593767487448558194897009318829335017531312}}$\\
$xyz^{3}$ & $W$ & $\frac{\mathtt{63124143098168155026060265986978624933421514365737963701974316645275}}{\mathtt{220429674396237054488595305593767487448558194897009318829335017531312}}$\\
$xz^{4}$ & $U$ & $\mathtt{0}$\\
$xz^{4}$ & $W$ & $\mathtt{0}$\\
$y^{5}$ & $U$ & $\frac{\mathtt{186871683168118003488935611555116889880084010}}{\mathtt{32767303919729427758399266748957493087161279}}$\\
$y^{5}$ & $W$ & $\frac{\mathtt{25444070085364078630942054472628302788437558}}{\mathtt{32767303919729427758399266748957493087161279}}$\\
\ExactLedgerNext
$y^{4}z$ & $U$ & $\mathtt{2}$\\
$y^{4}z$ & $W$ & $\mathtt{0}$\\
$y^{3}z^{2}$ & $U$ & $\mathtt{0}$\\
$y^{3}z^{2}$ & $W$ & $\mathtt{0}$\\
$y^{2}z^{3}$ & $U$ & $\mathtt{0}$\\
$y^{2}z^{3}$ & $W$ & $\mathtt{0}$\\
$yz^{4}$ & $U$ & $\mathtt{0}$\\
$yz^{4}$ & $W$ & $\mathtt{0}$\\
$z^{5}$ & $U$ & $\mathtt{0}$\\
$z^{5}$ & $W$ & $\mathtt{0}$\\
\ExactLedgerTail
\end{landscape}

The exact canonical-map identity on $C_+$ is
\begin{equation}\label{eq:t-inverse-identity}
                         G_t+tH_t=0,
\end{equation}
and therefore $t=-G_t/H_t$ wherever $H_t\ne0$.
This rational function extends uniquely to a morphism
$C_+\to\mathbf P^1$, since $C_+$ is smooth and projective.
At a common zero of $G_t,H_t$, the value is obtained after cancelling
their common vanishing in a local parameter; it cannot be read from
the displayed projective pair alone.  The finite sieve retains such
common-zero reductions conservatively.  The certificate verifies
\eqref{eq:t-inverse-identity}; the five reference points are not common
zeros, and direct substitution gives
$t(P_1)=0$, $t(P_2)=1$, and $t(P_3)=t(P_4)=t(P_5)=\infty$.
Local expansions on the smooth quartic give
\[
 \ord_{P_1}(t)=5,\qquad \ord_{P_2}(t-1)=3,\qquad
 \ord_{P_3}(1/t)=\ord_{P_4}(1/t)=7,\qquad \ord_{P_5}(1/t)=1.
\]
In particular, $P_5$ is an unramified point above the branch value
$\infty$.

\section{Integral coordinates for the genus-two formal group}
\label{app:flynn}

On the divisor-pair chart, write
$\alpha_1=x_1+x_2$ and $\alpha_2=x_1x_2$.  The unnormalized rational
coordinate functions in Flynn's projective embedding include
\[
 a_{15}=1,\quad a_{14}=\alpha_1,\quad a_{13}=\alpha_2,\quad
 a_{12}=\alpha_1^2-2\alpha_2,\quad
 a_{11}=\alpha_1\alpha_2,\quad a_{10}=\alpha_2^2.
\]
These formulas use the representative with $a_{15}=1$; they do not
normalize $a_0$ to $1$.  Near the projective origin, the normalized
coordinates are instead $s_i=a_i/a_0$.  In particular,
$a_0a_{15}=a_5^2$, so $s_{15}=s_5^2$ and $s_{15}$ vanishes at the origin.

The computations use the following integral coordinate block:
\[
 b_i=a_i\quad(0\leq i\leq11),\qquad
 b_{12}=a_{13},\quad b_{13}=a_{14},\quad b_{14}=a_{15},\quad
 b_{15}=a_{12}-2a_{13}.
\]
Its inverse is $a_{12}=b_{15}+2b_{12}$, with all other displayed
coordinates recovered directly.  Hence this is an integral unimodular
change of basis.  Membership in the depth-two origin neighbourhood is
therefore equivalent in the $a$- and $b$-coordinates.  The valuation rows
in \eqref{eq:section7-valuations} are stated in the $b$-coordinates.

\section{The split-prime data at 23 for the rational-parameter theorem}
\label{app:p23-data}

The prime $23$ splits in $K_7$ as the two embeddings $s\mapsto4$ and
$s\mapsto19$.  Both reductions of $C_+$ have $20$ rational points.  Their
common local numerator of the zeta function is
\[
 1-4T+4T^2+56T^3+92T^4-2116T^5+12167T^6,
\]
so
\[
                      \#J_+(\F_{23})=10200,\qquad 23\nmid10200.
\]
The ordered reductions of the five points are
\[
\begin{array}{c|cc}
&s=4&s=19\\ \hline
P_1&(1,21,1)&(1,11,15)\\
P_2&(1,0,22)&(1,8,14)\\
P_3&(1,9,0)&(1,13,0)\\
P_4&(1,1,0)&(1,9,0)\\
P_5&(1,11,21)&(1,6,9).
\end{array}
\]
For the regular differential basis $\eta_x,\eta_y,\eta_z$, the corresponding
evaluation vectors are
\[
\begin{array}{c|cc}
&s=4&s=19\\ \hline
P_1&(3,17,3)&(11,6,4)\\
P_2&(4,0,19)&(5,17,1)\\
P_3&(4,13,0)&(2,3,0)\\
P_4&(12,12,0)&(22,14,0)\\
P_5&(7,8,9)&(10,14,21).
\end{array}
\]
These data independently reproduce the tangent determinants in
\eqref{eq:siksek-dets} once the two annihilator rows are applied.

\section{The complete rank-four descent and its downstream use}
\label{app:rank-proof}

Let $K=K_7=\Q(s)$, where $s^2=-7$, and let
$J=\operatorname{Jac}(C_+)$ for the smooth plane quartic of
Appendix~\ref{app:Cplus}.  Write $D_i=[P_i-P_1]$ and
\[
 H_0=\langle D_2,D_3,D_4,D_5\rangle,\qquad
 H_1=\langle D_2,E,D_4,D_5\rangle.
\]
We prove that $J(K)$ has rank four, that $H_0$ has odd index, and that
$[J(K):H_1]$ is coprime to $400$.  These are the group-theoretic inputs
to the sieve and logarithmic argument in Theorem~\ref{thm:rational-t}.
The relation \eqref{eq:Erelation} gives $H_0\subseteq H_1$; the
$2$-saturation statement concerns $H_0$, and the $5$-saturation statement
concerns $H_1$.

\subsection{The descent maps and the local completeness criterion}

Let $\Delta$ be the Galois set of the $28$ bitangents of $C_+$, and let
$L=(\overline K^{\Delta})^{\operatorname{Gal}(\overline K/K)}$
be its degree-$28$ \'etale algebra of Galois-equivariant functions.  Evaluation of
the normalized bitangent lines on a degree-zero divisor $z$ gives
$f(z)\in L^\times$, whenever the divisor avoids their zeroes and poles.
For $F=K$ or a completion $K_v$, put
\[
 \mathcal H(F)=
 (L\otimes_K F)^\times\big/
       \bigl(F^\times(L\otimes_K F)^{\times2}\bigr).
\]
The evaluation construction induces the fake descent map
\[
 f_F:J(F)/2J(F)\longrightarrow\mathcal H(F).
\]
Thus a divisor class is in the \emph{fake kernel} precisely when its
evaluation can be written $f(z)=a\lambda^2$, with $a\in F^\times$
and $\lambda\in(L\otimes_K F)^\times$.

For clarity, the finite modules in the comparison with true descent are
as follows.  Let $E_\Delta\subset\F_2^\Delta$ be the subspace of even subsets,
and let $R_\Delta\subset E_\Delta$ be the span of the $315$ syzygetic tetrads.
Write $E_\Delta^\vee,R_\Delta^\vee$ for their duals and
$q:E_\Delta^\vee\to R_\Delta^\vee$ for restriction.  The exact bitangent and
contact identities identify
\[
 \ker q\simeq J[2],\qquad \dim E_\Delta=27,\qquad \dim R_\Delta=21,
\]
as Galois modules, by \cite[Corollary~12.5]{BPS}.  The rational even
theta characteristic and the seven blocks in the exact field tower put
the actual Galois action inside $\operatorname{AGL}_3(\F_2)$, acting
transitively on $\Delta$.  Equality with the full affine group is not
needed.  The rational point $P_1$ supplies the degree-one divisor
hypotheses for the global and local comparison results of
\cite[Lemma~10.2 and Appendix~A]{BPS}.

If $M$ is one of these finite Galois modules, $M(F)$ denotes its
$F$-rational elements.  Define the correction quotient
\[
 \mathcal C(F)=R_\Delta^\vee(F)/q\bigl(E_\Delta^\vee(F)\bigr).
\]
The tetrad incidence map $\tau_*$ multiplies the four corresponding
bitangent values.  The normalized contact conics give an evaluation
$r(z)$ satisfying the exact identity
\[
                         r(z)^2=\tau_*(f(z)).
\]
Consequently, if $f(z)=a\lambda^2$, the correction of $z$ is represented by
\begin{equation}\label{eq:rank-kernel-correction}
             \frac{r(z)}{a^2\tau_*(\lambda)}.
\end{equation}
The scalar exponent is two because each tetrad has four entries.
Formula \eqref{eq:rank-kernel-correction} is the identity of
\cite[Lemma~A.24(a)]{BPS}.  Its entries lie in $\mu_2$ before any
local approximation is made.  A nonzero value under a character of
$\mathcal C(F)$ therefore proves that the class of $z$ is nonzero in
$\ker f_F$.

\begin{lemma}[Local completeness]\label{lem:rank-local-completeness}
Let $d_v=\dim_{\F_2}J(K_v)/2J(K_v)$.  Suppose actual local divisor
classes give $r_v$ independent fake images and $k_v$ independent classes
in the fake kernel.  If $r_v+k_v=d_v$, these images span the full fake
image and the fake kernel has dimension $k_v$.
\end{lemma}

\begin{proof}
Apply rank--nullity to $f_{K_v}$.  The two exhibited subspaces give
lower bounds whose sum already equals the dimension of its domain.
For a finite place of odd residue characteristic,
$d_v=\dim J[2](K_v)$; for $v\mid2$,
$d_v=3[K_v:\Q_2]+\dim J[2](K_v)$, by
\cite[Remark~11.6]{BPS}.  This is the stopping criterion of
\cite[Lemma~A.24(b)]{BPS}.
\end{proof}

\subsection{A containing space for the global fake image}

Let $S$ consist of the complex place, both dyadic places, the places above
$3,5,7$, and the bad place $(29,s-15)$.  The primitive quartic is
smooth outside $S$, and the bitangent algebra is unramified there.
The base field $K$ has class number one.  The containing direction of
\cite[Propositions~7.2--7.3 and Theorem~10.9(a)]{BPS} therefore puts
every relevant fake Selmer class in the image of $L(S,2)$, the
squareclasses with even valuation outside the primes above $S$.

The class-group input is the following lemma; the earlier text
compressed it into one sentence, and its full chain is recorded in the
companion as indicated.

\begin{lemma}\label{lem:cl2}
$\operatorname{Cl}(L)[2]=0$.
\end{lemma}

\begin{proof}[Proof sketch]
The degree-$56$ field $L$ has an $S_4$-type Galois closure over the
degree-$14$ field $E$; write $F$ for the corresponding degree-$42$
field, $N$ for the degree-$84$ field with $N/F$ quadratic, and $T$ for
the degree-$21$ subfield with $F=T(\sqrt{-7})$.  The Brauer--Kuroda
relation $\zeta_E\zeta_N=\zeta_L\zeta_F$ (all $24$ permutation-character
identities are verified exactly) and the exact regulator comparison
give
\[
 h(L)=\frac{h(E)\,h(N)}{h(F)}\cdot\frac{2^{27}}{J},
\]
where $J$ is the index of the unit correspondence; exact
$2$-saturation of that correspondence gives $v_2(J)=27$.  It therefore
suffices that $h(E)$, $h(F)$ and $h(N)$ are odd.
$h(E)=1$ is certified unconditionally (\texttt{bnfcertify}).
For $T$: any unramified quadratic extension of $T$ split at the real
places has discriminant $d_T^2$ and splits every principal prime of $T$;
the $3{,}216$ certified principal primes of norm at most $30\,000$ enter
the unconditional local-corrected discriminant bound of
Brueggeman--Doud \cite[Theorem~2.4(1)]{BrueggemanDoud}, evaluated with
interval arithmetic, and the bound exceeds $\log d_T^2$ by a margin of
$0.0407$; hence no such extension exists and $h(T)$ is odd.
$F/T$ is ramified at one finite and three real places with unit
Hilbert-symbol rank $3$, and $N/F$ is ramified at four primes with unit
symbol rank $3$; the ambiguous-class formula gives $h(F)$ and then
$h(N)$ odd.
\begingroup\raggedright
The steps are the prior-evidence jobs
\texttt{E\_class\_one}, \texttt{T21\_prime\_principality},
\texttt{T21\_norm\_export}, \texttt{T21\_directed\_Arb\_bound},
\texttt{F\_odd\_class}, \texttt{pair\_ambiguous},
\texttt{J27\_images}, \texttt{J27\_chain}, \texttt{J27\_terminal},
\texttt{S4\_integral\_relation} and \texttt{class\_field\_bindings} of
the companion.  Under \path{evidence/v3/repository/results/}, the
relevant runs are:
\begin{itemize}\small
\item \path{2026-08-26_l2_relative_s4_class_gate}: the acceptance
contract, whose initial status was \texttt{NO\_GO};
\item \path{2026-08-26_t_odd_odlyzko_exact_local_certificate_v1}: the
discriminant certificate;
\item \path{2026-08-26_t_odd_odlyzko_downstream_implication_replay_v1}:
the completed conditions of the contract
(\path{summaries/T_ODD_DOWNSTREAM_IMPLICATION_REPLAY_V1.md},
\path{certificates/downstream_implication_decision_v1.json}).
\end{itemize}
\endgroup
\end{proof}
The $38$ finite primes of $L$ above $S$ are principal, and the unit calculation
gives $28$ independent ordinary-unit squareclasses, including $-1$.
The exact valuation sequence for $S$-unit squareclasses shows that
these $28$ classes and the $38$ prime generators form a basis of the
$66$-dimensional space
\[
                    V_0=L(S,2)\subset L^\times/L^{\times2}.
\]
Thus every fake image of a true $2$-Selmer class has a representative
in $V_0$.
For a selected place $v$, impose on $V_0$ the condition that its
restriction modulo the diagonal $K_v^\times$ belong to
$\operatorname{im}f_{K_v}$.

The local image dimensions used here are
\[
\begin{array}{c|c|c}
 \text{place}&d_v&\dim\operatorname{im} f_{K_v}\\ \hline
 \text{four completely split bitangent places}&6&6\\
 v_{1051}=(1051,s-815)&4&4\\
 v\mid3&2&2\\
 v\mid5&3&3\\
 v\mid7&2&1\\
 v\mid2,\ w\equiv1\pmod2&4&4\\
 v\mid2,\ w\equiv0\pmod2&4&3
\end{array}
\qquad w=(1+s)/2.
\]
At the four good places above $19163,23431,29059,38189$, the bitangent
algebra splits completely.  Although $1051$ also splits in $K/\Q$,
at the selected good place $v_{1051}=(1051,s-815)$ the bitangent algebra
has residue factor degrees $1^8 2^{10}$.  Thus
$L\otimes_K K_{v_{1051}}\simeq K_{v_{1051}}^8\times E_{1051}^{10}$, where
$E_{1051}/K_{v_{1051}}$ is the unramified quadratic extension.  Its Frobenius
action gives $\dim_{\F_2}J[2](K_{v_{1051}})=4$.
For the first four rows, independent point images attain the local
two-torsion bound, so Lemma~\ref{lem:rank-local-completeness} applies with
$k_v=0$.  The $7$-adic and dyadic arguments below justify the remaining
rows, including their kernel contributions.

Let $V\subset V_0$ be the intersection of these conditions.  The
restriction of each of the $66$ generators is computed in every factor
of the actual local bitangent algebra.  Exact linear algebra gives
\begin{equation}\label{eq:rank-dimensions}
 66\longrightarrow23\longrightarrow22\longrightarrow20
 \longrightarrow19\longrightarrow18\longrightarrow15\longrightarrow12.
\end{equation}
The stages are the four completely split bitangent places together,
$v_{1051}$, the places above
$3,5,7$, and the dyadic places with $w\equiv1,0$, in that order.
The other prime conditions and the norm condition in the ancillary
descent calculation are redundant for this intersection.

The diagonal subspace
\[
 D=V\cap\operatorname{im}\bigl(K^\times/K^{\times2}
                        \longrightarrow L^\times/L^{\times2}\bigr)
\]
has dimension seven.  Set $Q=V/D$
and let $B$ be the fake image of $H_0$.  The coordinate calculation gives
\begin{equation}\label{eq:rank-containing-quotient}
 Q=B\oplus\langle c\rangle,\qquad \dim B=4,\qquad \dim Q=5.
\end{equation}
Here $c$ denotes the fixed complementary squareclass in the descent
data; we use the same representative in $L^\times$ throughout the
comparison below.  The calculation verifies that the diagonal,
the four point images, and this representative satisfy every local
condition, have ranks $7,11,12$ successively, and span $V$.
Thus \eqref{eq:rank-containing-quotient} is an equality of the specified
spaces.  In particular, the fake image of every true Selmer class lies
in $Q$.  We do not need every element of $Q$ to be a Selmer class.

\subsection{Completeness at seven}

Let $v\mid7$.  The actual decomposition data give $d_v=2$.
For the divisor $z_7=P_4-P_1$, exact square tests in all seven factors
of $L\otimes_K K_v$ give
\[
                            f(z_7)=3\lambda^2.
\]
This puts its class in the full fake kernel.  To prove it is a nonzero
class, consider the specified completion $A$ of the degree-$42$ tetrad
algebra.  It has $(e,f)=(4,1)$ over $K_v$, and the incidence algebra
over $A$ is a product of two unramified quadratic extensions.
Evaluation on this tetrad component defines a character of
$\mathcal C(K_v)$: this is checked for every affine-group candidate
compatible with the measured decomposition data.  Some of those
characters vanish, so their factor profiles alone do not prove
nonvanishing.

In the two quadratic incidence fields, remove the even valuations from
$f(z_7)/3$ and take square roots of the resulting units.  Their residue
roots over $\F_{49}$ are simple and hence lift by Hensel's lemma.
The two normalized norm residues are one; the corresponding normalized
conic quotient has residue six.  Formula
\eqref{eq:rank-kernel-correction} therefore gives
\[
           \frac{r(z_7)}{3^2\tau_*(\lambda)}=-1
\]
on the selected target.  The equality is exact: the contact identity
puts the quotient in $\mu_2$, and reduction in residue characteristic
seven distinguishes $1$ from $-1$.  Thus $\ker f_{K_v}$ is nonzero.
For comparison, the distinct divisor $P_2-P_1$, also with diagonal
scalar $3$, gives $+1$.  Changing either quadratic square-root sign
does not change its norm.

There is also a nonzero fake image.  Let $K'_v/K_v$ be the unramified
quadratic extension.  The local-point calculation produces
$P\in C_+(K'_v)$ such that the degree-zero divisor
\[
                   \operatorname{Norm}_{K'_v/K_v}(P)-2P_1
\]
has nonzero fake image.  The coordinate change from its local model to
the printed quartic is $Z_{\rm original}=-7Z_{\rm local}$.
A strong Hensel inequality proves existence of the point, and all seven
line-value squareclasses are unchanged by the correction.  Its norm-line
vector, modulo the two diagonal classes, spans the one-dimensional
allowed image used in \eqref{eq:rank-dimensions}.
We have now exhibited an image of dimension at least one and a kernel
of dimension at least one in a domain of dimension two.
Lemma~\ref{lem:rank-local-completeness} proves that both dimensions are
exactly one.

\paragraph{Precision for the $7$-adic calculation.}
The target is represented by an integral Eisenstein polynomial of
degree eight over $\Q_7$.  Strong Hensel bounds and Krasner's lemma
identify the computed generator with the exact target root; Gauss
bounds propagate its error through the rational line, incidence and
conic coefficients, including their denominators.  The two incidence
factors lift from coprime factors.  A uniformizer change in one factor
gives its irreducible residual quadratic.  At working precision $100$
over $\Q_7$, the resulting relative-error bound is at least $254$ in
$A$, so the residue used above is unchanged.  The computed sign
separations at working precisions $100$ and $120$ are $255$ and $511$.
These inequalities justify passage from approximations to the exact
sign, rather than relying on agreement of two numerical outputs.

\subsection{Completeness at the dyadic places}

At both dyadic places $d_v=4$.  At the place with $w\equiv1\pmod2$,
the known point differences together with the Hensel lift of
$[2:46:1]$ give four independent fake images.  This already attains
$d_v$, so the image is complete and its fake kernel is zero.

Now work at the place with $w\equiv0\pmod2$, identified with $\Q_2$.
The four known differences have fake-image rank two.  We describe the
additional points and kernel divisors explicitly.  In the chart
\[
                       [X:Y:Z]=[13+32e:1:16d],
\]
let $C_0=[13:1:16d_0]$, where $d_0\equiv0\pmod2$ is the Hensel
root of the quartic equation.  The normalized equation reduces to
$d(d+1)^3$, so this root is unique in its residue disk.  Its fake image
is independent of the known images, giving rank at least three.

In the chart $[13+32e:1:16(1+2g)]$, the normalized reduction is
\[
       e^4+e^3+g^2e^2+(g^2+g)e+g^3+g^2+g=0.
\]
Its $g$-derivative at $(0,0)$ and its $e$-derivative at $(1,0)$ are
units.  Hence the following points are well defined by their coordinates
and congruences:
\[
 \begin{aligned}
 P'_2&=[77:1:16(1+2g_2)],&g_2&\equiv0\pmod2,\\
 P'_5&=[13+32e_5:1:80],&e_5&\equiv1\pmod2.
 \end{aligned}
\]
Here $g_2$ is obtained by fixing $e=2$, and $e_5$ by fixing $g=2$.
Define the degree-zero divisors
\[
 z_2=P'_2+P_2+C_0-3P_1,\qquad
 z_5=P'_5+P_2-2P_1.
\]
Exact square tests in all three factors of the local bitangent algebra
give
\[
                    f(z_2)=-10\lambda_2^2,\qquad
                    f(z_5)=2\lambda_5^2.
\]
Thus these are actual fake-kernel divisors.  The line-value error bounds
put the ratio of each true value to its tested approximation in
$1+\mathfrak p^{2\ord_{\mathfrak p}(2)+1}$, whose elements are squares
by Hensel's lemma.

The degree-$42$ target splits here into factors of degrees
$3,3,6,6,24$.  Evaluation on the two sextic factors gives two independent
characters of the two-dimensional quotient $\mathcal C(K_v)$.
Both characters annihilate $q(E_\Delta^\vee(K_v))$ for all $28$ compatible
affine-group candidates.  With $-1$ represented by the bit $1$, both
$z_2$ and $z_5$ have correction values $(0,1)$ under these characters.
In particular the fake kernel is nonzero.  The rank-three image and
this kernel class fill the four-dimensional local Kummer group;
Lemma~\ref{lem:rank-local-completeness} proves that the full fake image
has dimension three.  This proves the last two cuts in
\eqref{eq:rank-dimensions}.

\paragraph{Precision for the dyadic correction.}
Integral local bases, ramification and residue polynomials identify the
two sextic fields and their quadratic incidence factors.  Strong
Hensel and Krasner bounds apply to the original target polynomial,
and the factor-lifting error is propagated through the evaluations.
In the strengthened calculation for $z_5$, the unit square-root
residual and derivative valuations are $(40,4)$ in a ramified quadratic
factor and $(20,2)$ in an unramified one.  The corresponding exact-root
error bounds are $36$ and $18$; after taking norms the relative error
in the target is at least $18$.  The final sign-separation pairs are
$(19,2)$ and $(2,18)$; in each sextic target field $A$,
$\ord_A(2)=2$.  The exact contact identity
therefore distinguishes the signs despite residue characteristic two.
Only a valid square-root witness is required: changing its sign has
no effect on a quadratic norm, so uniqueness of the finite root search
is unnecessary.

\subsection{Excluding the complementary class}

We next compare the containing space $Q$ with the true $2$-Selmer group.
In fact, $J(K)[2]=0$ follows already from transitivity on $\Delta$.
Identify $E_\Delta^\vee$ with
$\F_2^\Delta/\langle\mathbf1\rangle$ using the standard dot product.
The bitangent realization gives
$J[2]\simeq R_\Delta^\perp/\langle\mathbf1\rangle$.
The weight enumerator of $R_\Delta^\perp$ is
\[
                         1+63T^{12}+63T^{16}+T^{28};
\]
it is obtained by taking the orthogonal complement of the $315$ tetrad
incidence vectors, which have rank $21$, and enumerating its $128$ vectors.
Each nonzero quotient class therefore has two complementary
representatives of weights $12$ and $16$.  If such a class were Galois
invariant, its unique weight-$12$ representative would have
Galois-invariant support, a nonempty proper subset of the transitive set
$\Delta$.  This is impossible.  Separately, the finite Galois-module
calculation gives $\dim\mathcal C(K)=1$ for every transitive affine
subgroup compatible with the actual bitangent action.

The fixed representative $c$ in \eqref{eq:rank-containing-quotient}
has two coherent global lifts under the BPS comparison map on
$H^1(K,J[2])$.  Here coherence includes the relations required by
\cite[Lemma~A.13(b)]{BPS}, rather than only the existence of square roots
of $\tau_*(c)$.  More explicitly, the degree-$42$ and degree-$56$
target orbits give $98$ incidence vectors spanning $R_\Delta$.  Exact target
norm roots satisfy the $77$ independent relations among those vectors.
Their relation quotients lie in $\mu_2$, so their exact signs are
determined by reduction at the odd split prime $733177$, where all
denominators are units.  This defines the lift on all $315$ tetrads
and respects the Galois action.  The two compatible whole-orbit sign
choices, denoted $00$ and $11$, differ by the nonzero element
$\rho\in\mathcal C(K)$.  Since this quotient has dimension one,
these are all the lifts.

At $v\mid5$ the fake map is injective: the three independent images
already attain $d_v=3$.  Hence the local correction quotient used to
test whether a cohomology class is Kummer is $\mathcal C(K_v)$ itself;
there is no local fake-kernel contribution to quotient out.  Its
dimension is two.  For the actual divisor $z=P_2+P_4-2P_1$ and
$a=(s-1)/2$, the same global representative satisfies
\[
                             f(z)=ca\lambda^2
\]
in all eleven components of $L\otimes_K K_v$.
If $u^2=\tau_*(c)$ is a coherent target root, the local obstruction
is represented by
\[
                         \frac{u\,a^2\tau_*(\lambda)}{r(z)}.
\]
Its square is one by the exact contact and norm identities.  The
physical target characters give
\[
 \omega_{00}=(1,0),\qquad \omega_{11}=(0,1),\qquad
                         \kappa_5(\rho)=(1,1).
\]
Here $\kappa_5$ is localization on the correction quotient; its value
is the difference of the two obstructions, not either obstruction
itself.  The functional $(u,v)\mapsto u+v$ vanishes on that difference
and is nonzero on both obstructions.  Thus neither lift of $c$ is
locally Kummer at $5$.

By \cite[Theorem~10.14 and Lemma~A.26]{BPS}, the kernel of the
true-to-fake Selmer comparison is the kernel of localization from
$\mathcal C(K)$ to the local correction quotients.  Its restriction
at $5$ is already injective, since $\kappa_5(\rho)\ne0$.
The Selmer comparison kernel is consequently zero.

\paragraph{Arithmetic identification of the $5$-adic obstructions.}
The exact coordinate representative of $c$, the target norm roots,
and the local source equation above are compared directly.  In the
selected quadratic incidence fields, taking norms of $\lambda^2$
and prescribing the residue of the norm of $\lambda$ determines the
required square root.  The three physical character values are
$(-1,+1,-1)$; the first and third characters agree, and the first two
are independent.  Reversing the coherent root gives $(+1,-1,+1)$.
The relative root-error bounds are at least $21$, so these residues
determine the exact $\mu_2$ values.  This is the arithmetic underlying
the two displayed obstruction vectors.

\subsection{Rank, subgroup indices and the logarithmic annihilator}

\begin{proposition}
The true $2$-Selmer group and the Mordell--Weil group satisfy
\[
 \dim_{\F_2}\operatorname{Sel}^{2}(J/K)=\rank J(K)=4,
                         \qquad \Sha(J/K)[2]=0.
\]
\end{proposition}

\begin{proof}
Every true Selmer class has fake image in $Q=B\oplus\langle c\rangle$.
If its image were $c+b$, choose a point of $H_0$ with fake image $b$.
Subtracting its global Kummer class would give an everywhere locally
Kummer lift of $c$, contradicting the obstruction at $5$.
Thus the true Selmer image lies in $B$.  The comparison is injective,
so the Selmer dimension is at most four.  The independent point lower
bound is four.  Since $J(K)[2]=0$, the Kummer exact sequence
\[
 0\longrightarrow J(K)/2J(K)
 \longrightarrow\operatorname{Sel}^{2}(J/K)
 \longrightarrow\Sha(J/K)[2]\longrightarrow0
\]
proves all three assertions.
\end{proof}

The verified $2$-saturation, equivalently the independence of the four
generators modulo $2$ together with this rank calculation, gives
$[J(K):H_0]$ odd.  The relation \eqref{eq:Erelation} gives
$H_0\subseteq H_1$, so $H_1$ also has odd index.  Its separately
verified $5$-saturation then gives
\[
                         \gcd([J(K):H_1],400)=1.
\]
Multiplication by $400$ is therefore an automorphism of the finite
group $J(K)/H_1$, which proves that
$H_1\to J(K)/400J(K)$ is surjective.  This is precisely the coefficient
coverage required by the finite sieve; saturation at other primes
is not used.

Finally, let a characteristic-zero logarithmic functional vanish on
$H_0$.  Every point of $J(K)$ has a nonzero integer multiple in $H_0$,
and the target of the functional has no integer torsion.  The
functional therefore vanishes on all of $J(K)$.  This proves that the
logarithmic annihilator used in Theorem~\ref{thm:rational-t} is global.

\paragraph{Location of the arithmetic evidence.}
The ancillary component guide \path{PROGRAMS_AND_CERTIFICATES.md}
identifies the ordered $66$ representatives, all restriction columns,
the diagonal and point coordinates, the fixed representative $c$,
and the exact root and local-divisor witnesses used here.  It also
locates the full local-field and precision calculations behind the
inequalities above.  These data establish the computational premises
of the argument; checksums and saved status messages serve only to
identify and reproduce them.

\end{document}